\documentclass[a4paper,11pt,reqno]{amsart} 
\usepackage{amsmath}
\usepackage{amsfonts}
\usepackage[T1]{fontenc}
\usepackage[utf8]{inputenc}
\usepackage[mathscr]{euscript}
\usepackage{verbatim}
\usepackage{amssymb,latexsym}
\usepackage{amsthm}
\usepackage{eufrak}
\usepackage{color}
\usepackage[lmargin=2.5 cm,rmargin=2.5 cm,tmargin=3.5cm,bmargin=2.5cm,paper=a4paper]{geometry}

\newtheorem{thm}{Theorem}[section]
\newtheorem{pro}[thm]{Proposition}

\newtheorem{lem}[thm]{Lemma}
\newtheorem{cor}[thm]{Corollary}

\theoremstyle{remark}

\newtheorem{rem}[thm]{Remark}

\newcommand{\Int}[2]{\displaystyle{\int_{#1}^{#2}}}

\newcommand{\R}{\mathbb{R}}

\newcommand{\norm}[1]{\left\|#1\right\|}

\def\sig#1{\vbox{\hsize=5.5cm
\kern2cm\hrule\kern1ex
\hbox to \hsize{\strut\hfil #1 \hfil}}}

\newcommand\signatures[4]{%
\vspace{3cm}
\hbox to \hsize{\hfil #1, \today\hfil}
\vspace{3cm}
\hbox to \hsize{\quad#2\hfil\hfil #3\quad}
\vspace{3cm}
\hbox to \hsize{\hfil#4\hfil}}
\numberwithin{equation}{section}

\title[Robin Laplacian]{Eigenvalues  for  the Robin Laplacian in domains with variable curvature}
\author[B. Helffer]{Bernard Helffer}
\address[B. Helffer]{Universit\'e de Paris-Sud, B\^at 425, 91405 Orsay, France and Laboratoire Jean Leray, Universit\'e de Nantes, France.}
\email{bernard.helffer@math.u-psud.fr}
\author[A. Kachmar]{Ayman Kachmar}
\address[A. Kachmar]{Lebanese University, Department of Mathematics, Hadath, Lebanon.}
\email{ayman.kashmar@gmail.com}
\begin{document}
\thispagestyle{empty}

\newpage
\maketitle
\begin{abstract}
We determine accurate asymptotics for the low-lying eigenvalues of
the Robin Laplacian when the Robin parameter goes to $-\infty$. The
two first terms in the expansion have been obtained by K.
Pankrashkin in the $2D$-case and  by K. Pankrashkin and  N. Popoff
in higher dimensions. The asymptotics display the influence of the
 curvature and the splitting between every two consecutive
eigenvalues. The analysis is based on the approach developed by
Fournais-Helffer for the semi-classical magnetic Laplacian.  We also
propose a WKB construction as candidate for the ground state energy.
\end{abstract}

\section{Introduction}
Let $\Omega\subset\R^{2}$ be an open domain with a smooth $C^\infty$
and compact boundary $\Gamma=\partial\Omega$ such that the  boundary
$\Gamma$ has a finite number of connected components. The  unit
outward normal vector  of the boundary $\partial\Omega$ is denoted
by $\nu$.

In this paper, we study the low-lying eigenvalues of the Robin
Laplacian in $L^2(\Omega)$ with a large parameter. This is the
operator
\begin{equation}\label{Shr-op-Gen}
\mathcal{P}^\gamma=-\Delta\quad{\rm in~}L^2(\Omega),
\end{equation}
with domain,
\begin{equation}\label{eq:bc}
D(\mathcal P^\gamma)=\{u\in H^2(\Omega)~:~\nu\cdot\nabla u+\gamma\,u=0\quad{\rm on}~\partial\Omega\}\,,
\end{equation}
where $\gamma<0$ is a given parameter. The operator $\mathcal
P^\gamma$ is  defined by the Friedrichs Theorem via the closed
semi-bounded quadratic form, defined on $H^1(\Omega)$ by
\begin{equation}\label{QF-Gen}
u\mapsto  \mathcal{Q}^\gamma(u):=\norm{\nabla u}^{2}_{L^{2}(\Omega)}+\gamma\Int{\partial\Omega}{}|u(x)|^{2}ds(x)\,.
\end{equation}
Denote by $(\lambda_n(\gamma))$ the sequence of min-max values of
the operator $\mathcal P^\gamma$. In \cite{Pan, PanP, EMP}, it is
proved that, for every fixed $n\in\mathbb N$,
\begin{equation}\label{eq:pan}
\lambda_n(\gamma)=-\gamma^2+\kappa_{\max}\gamma+\gamma\, o(1)\quad{\rm
as~}\gamma\to-\infty\,,\end{equation} where $\kappa_{\max}$ is the
maximal  curvature along the boundary $\Gamma$. Note that the
first term in \eqref{eq:pan} was obtained previously (see \cite{LP}
and references therein).

 If the domain $\Omega$ is an exterior
domain, then the operator $\mathcal{P}^\gamma$ has an essential
spectrum; the essential spectrum is $[0,\infty)$, see e.g.
\cite{KP}. In this case, the asymptotics in \eqref{eq:pan} show
that, for every fixed $n$,  if $-\gamma$ is selected sufficiently
large, then $\lambda_n(\gamma)$ is in the discrete spectrum of the
operator $\mathcal P^\gamma$.  When the domain $\Omega$ is an
interior domain, i.e. bounded, then by Sobolev embedding, the
operator $\mathcal{P}^\gamma$ is with compact resolvent and its
spectrum is purely discrete.

The aim of this paper is to improve the asymptotic expansion in
\eqref{eq:pan}  and to give the leading term of the spectral gap
$\lambda_{n+1}(\gamma)-\lambda_n(\gamma)$. We will do this under a
generic assumption on  the domain $\Omega$ that we describe below.

Suppose that the boundary $\partial\Omega$ is parameterized by
arc-length and $\kappa$ is the  curvature of $\partial\Omega$.
Let $\kappa_{\max}$ denotes the maximal value of $\kappa$. We choose
the arc-length parametrization of the boundary such that
$\kappa(0)=\kappa_{\max}$. {\bf Throughout this paper, we suppose
that:}
$$
Assumption~(A)\qquad\left\{
\begin{array}{l}
\kappa \text{ \rm  attains its maximum } \kappa_{\max} \text{ \rm at a unique point}\,;  \\
\text{\rm the maximum is non-degenerate, i.e. } k_2:=-\kappa''(0)>0.
\end{array}
\right.
$$
The main result in this paper is:

\begin{thm}\label{thm:HK}
For any positive $n$, there exists a sequence {$(\beta_{j,n})_{j\geq0}$, such  that, for any positive $M\in\mathbb N$}, the eigenvalue
$\lambda_n(\gamma)$ has,  as $\gamma\to-\infty$, the  asymptotic
expansion
$$ 
\lambda_n(\gamma)=-\gamma^2+\gamma \kappa_{\max}+(2n-1)\sqrt{\frac{k_2}2}\,|\gamma|^{1/2}+\sum_{j=0}^M\beta_{j,n}|\gamma|^{\frac{-j}{2}}+|\gamma|^{\frac{-M}2}o(1)
\,.$$
\end{thm}
%
%
%
Compared with  the existing semi-classics of Schr\"odinger operators
(see \cite{HSj}), the  curvature in Theorem~\ref{thm:HK} acts
as a potential well. The assumption $(A)$ indicates the case of a
unique well.  In the case considered here, one can  say that, in the limit $\gamma \rightarrow -\infty$, the bottom of the spectrum of the operator $\mathcal P^\gamma$ is determined by  the (boundary)  operator 
$H^{\rm bnd}:= - \frac{d^2}{ds^2}- \gamma^2   + \gamma  \kappa (s)\,$. Here $s$ is the arc-length parameterization of the boundary and the error through this approxomation is  $\mathcal O(1)$. Loosely speaking, the reduced (boundary) operator $H^{\rm bnd}$ is analyzed by replacing the potential (curvature) $-\kappa(s)$
by its quadratic approximation   at the (unique) point of  minimum.

As in \cite{HSj}, a natural and interesting question is
to discuss the case of multiple wells.  If we replace $(A)$ by
$$
(A')~\left\{
\begin{array}{l}
\kappa \text{ \rm  attains its maximum } \kappa_{\max} \text{ \rm at a finite number of points }\{s_0,s_1,\cdots,s_j\}\,;  \\
\text{\rm the maximum is non-degenerate, i.e. for all }i,~ k_{2,i}:=-\kappa''(s_i)>0\,.\\
\end{array}
\right.
$$
then many effects can appear depending on the values of the
$\kappa''(s_i)$ (as in the case of the Schr\"odinger operator). In
case of symmetries, the determination of the tunneling effect
between the points of maximal curvature is expected to play an
important role. A limiting situation is discussed in \cite{HPan}
when the domain $\Omega$ has two congruent corners (at a corner, we
can assign the value $\infty$ to $\kappa_{\max}$).  In the regular
case (typically an ellipse), an interesting step could be a
construction of WKB solutions in the spirit of a recent work by V.
Bonnaillie, F. H\'erau and N.~Raymond \cite{BHR}. We address this
WKB construction shortly in Section~\ref{sec:wkb}, and hope to come
back at this point in a future work.
\subsection*{Acknowledgements} The authors would like to thank K.
Pankrashkin for useful discussions. The first author is supported by the ANR programme NOSEVOL and the second author  is supported by a grant from
Lebanese University.
\section{Transformation into  a semi-classical problem}

We will prove Theorem~\ref{thm:HK} by transforming  it into  a
semi-classical problem as follows.  Let
$$h=\gamma^{-2}\,.$$
 The limit
$\gamma\to-\infty$ is now equivalent to the semi-classical limit
$h\to0_+$. Notice the simple identity
$$\forall~ u\in H^1(\Omega)\,,\quad \mathcal Q^\gamma(u)=h^{-2}\left(\int_\Omega|h\nabla
u|^ 2-h^{3/2}\int_{\partial\Omega}|u|^2\,ds(x)\right)\,.$$
We define the
operator \begin{equation}\label{eq:Lh} \mathcal
L_h=-h^2\Delta\,,\end{equation} with domain,
\begin{equation}\label{eq:Lh-dom}
D(\mathcal L_h)=\{u\in H^2(\Omega)~:~
{ \nu\cdot h^{1/2}\nabla u-u=0}
{\rm ~on~}\partial\Omega\}\,.\end{equation} Clearly,
$$ \sigma (\mathcal P^\gamma)=h^{-2}\sigma(\mathcal L_h)\,.$$
 Let
$(\mu_n(h))$ be the sequence of min-max values of the operator
$\mathcal L_h$. To determine the asymptotics of $\lambda_n(\gamma)$
as $\gamma\to-\infty$, we will determine the semi-classical
asymptotics of $\mu_n(h)$ as $h\to0_+$.

Now, Theorem~\ref{thm:HK} is a rephrasing of:

\begin{thm}\label{thm:HK'}
For any positive $n$, there exists a sequence 
{ $(\beta_{j,n})_{j\geq 0}$,}
 such that,  as $h\to0_+$, the eigenvalue $\mu_n(h)$
has for any positive 
{ $M\in\mathbb N$}
 the  asymptotic expansion
$$\mu_n(h)=-h-\kappa_{\max}h^{3/2}+(2n-1)\sqrt{\frac{k_2}2}\,h^{7/4}+h^{2}\sum_{j=0}^M
\beta_{j,n} h^{j/4}+h^{2+\frac{M}{4}}o(1)\,.$$
\end{thm}

The proof of Theorem~\ref{thm:HK'} consists of two major steps. In
the first step, we establish a  three-term asymptotics
$$\mu_n(h)=-h-\kappa_{\max}h^{3/2}+(2n-1)\sqrt{\frac{k_2}2}\,h^{7/4}+h^{7/4}o(1)\,,$$
which is valid under the weaker assumption that the boundary of the
domain $\Omega$ is $C^4$ smooth. This asymptotic expansion will be
established in Sections~\ref{sec:ub} and \ref{sec:lb}.

The next step is to construct good trial states and use the spectral
theorem to establish existence of $n$-eigenvalues of the operator
$\mathcal L_h$ satisfying the refined asymptotics
$$\widetilde\mu_n(h)=-h-\kappa_{\max}h^{3/2}+(2n-1)\sqrt{\frac{k_2}2}\,h^{7/4}+h^{15/8}\sum_{j=0}^M\zeta_{j,n} h^{j/8}+h^{\frac{M+15}{8}}o(1)\,,$$
which in light of the three-term asymptotics for $\mu_n(h)$, yields
the equality
$$\mu_n(h)=\widetilde\mu_n(h)\,,$$
for  $h>0$ sufficiently small. This analysis is presented in
Section~\ref{sec:Gr}. { Right after the construction of the sequence $(\zeta_{j,n})$, we will prove in Section~\ref{sec:zeta}  that   $\zeta_{j,n}=0$ when $j$ is even. The coefficients in Theorem~\ref{thm:HK} are given by the formula $\beta_{j,n}=\zeta_{2j+1,n}$.}

Throughout this paper, the notation $\mathcal O(h^\infty)$ indicates
a quantity satisfying that, for all $N\in\mathbb N$, there exists
$C_N>0$ and $h_N>0$ such that, for all $h\in(0,h_N)$, $|\mathcal
O(h^\infty)|\leq C_Nh^N\,$.   Examples of such quantities are
exponentially small quantities, like $\exp(-h^{-1/2})$.  The
letter $C$ denotes a positive constant independent of the parameter
$h$. The value of $C$ might change from one formula to another.

\section{Boundary coordinates}\label{sec:app}
The key to prove  Theorem~\ref{thm:HK'} is a reduction to the
boundary. Near the boundary, we do the computations using specific
coordinates displaying the arc-length along the boundary and the
normal distance to the boundary. In  this section, we introduce the
necessary notation to use these coordinates.

We will work in a connected component of the boundary. For
simplicity, we suppose that  $\Omega$ is simply connected; if
$\Omega$ is not simply connected, we use the coordinates in each
connected component independently.
Let
\[
\mathbb{R}/(|\partial \Omega |\mathbb{Z})\ni s\mapsto M(s)\in\partial\Omega
\]
 be a parametrization of $\partial\Omega$. The unit tangent vector of $\partial\Omega$ at the point $M(s)$ of the boundary is given by
\[
T(s):= M^{\prime}(s).
\]
We define the  curvature $\kappa(s)$ by the following identity
\[
T^{\prime}(s)=\kappa(s)\, \nu(s),
\]
where $\nu(s)$ is the unit vector, normal to to the
boundary, pointing outward at the point $M(s)$. We choose the
orientation of the parametrization $M$ to be counterclockwise, so
\[
\det(T(s),\nu(s))=1, \qquad \forall s\in \mathbb{R}/(|\partial \Omega |\mathbb{Z}) .
\]
For all $\delta>0$, define
\[
\mathcal{V}_{\delta}= \{x\in \Omega~:~{\rm dist} (x,\partial\Omega)<\delta\},
\]
 The map $\Phi$ is defined as follows~:
\begin{equation}
\Phi:\mathbb{R}/(|\partial \Omega |\mathbb{Z})\times(0,t_{0})\ni  (s,t)\mapsto x= M(s)-t\, \nu(s)\in \mathcal{V}_{t_{0}}.
\end{equation}
Alternatively, for $x\in\mathcal{V}_{t_{0}}$, one can write
\begin{equation}\label{BC}
\Phi^{-1}(x):=(s(x),t(x))\in { \mathbb{R}/(|\partial \Omega |\mathbb{Z})\times (0,t_{0})},
\end{equation}
where $t(x)={\rm dist}(x,\partial\Omega)$ and
$s(x)\in\mathbb{R}/(|\partial \Omega |\mathbb{Z})$ is the coordinate
of the point $M(s(x))\in\partial\Omega$ satisfying ${\rm
dist}(x,\partial\Omega)= |x-M(s(x))|$.

The determinant of the Jacobian of the transformation $\Phi^{-1}$ is
given by:
\[
a(s,t)=1-t\kappa(s).
\]
For all $u\in L^{2}(\mathcal V_\delta)$, define  the function
\begin{equation}
\widetilde u(s,t):= u(\Phi(s,t)).
\end{equation}
For all $u\in H^{1}(\mathcal{V}_{t_{0}})$, we have, with $\widetilde
u=u\circ \Phi$,
\begin{equation}\label{eq:bc;qf}
\int_{\mathcal{V}_{t_{0}}}|\nabla u|^{2}dx= \int \Big
[(1-t\kappa(s))^{-2}|\partial_{s} \widetilde u|^{2} +
|\partial_{t}\widetilde
u|^{2}\Big](1-t\kappa(s))dsdt\,, \end{equation}
and
\begin{equation}\label{eq:bc;n}
\int_{\mathcal{V}_{t_{0}}}|u|^{2}dx=\int|\widetilde u(s,t)|^{2}(1-t\kappa(s))dsdt.
\end{equation}

\section{Three auxiliary operators}\label{section4}
In this section, we introduce three reference $1D$-operators and
determine their spectra.

\subsection{1D Laplacian on the half line}\
 As simplest model, we start with the operator
\begin{equation}\label{defH00}
\mathcal H_{0,0}:=-\frac{d^2}{d\tau^2} \mbox{ in }  L^2(\R_+)
\end{equation}
with domain
\begin{equation}
\{u\in
H^2(\R_+)~:\,u'(0)=- u(0)\}\,.
\end{equation}
 The spectrum of this operator is
$\{-1\}\cup[0,\infty)$.
 This operator will appear in Sections \ref{sec:ub}, \ref{sec:Gr}
and \ref{sec:wkb}.

However, this operator has the drawback
that it is not with compact resolvent. To overcome this issue, we
consider this operator in a bounded interval $(0,L)$ with $L$
sufficiently big and Dirichlet condition at $\tau=L$. This operator
has a compact resolvent, and we will have fair knowledge about the
behavior of its spectrum in the limit $L\to\infty$.

\subsection{1D Laplacian on an interval}\ \\
Let $h>0$ and $\rho\in(0,1)$. Consider the self-adjoint operator
\begin{equation}\label{eq:H0}
\mathcal H_{0,h}=-\frac{d^2}{d\tau^2}\quad{\rm in~}L^2((0,h^{-\rho}))\,,
\end{equation}
with domain,
\begin{equation}\label{eq:DomH0}
D(\mathcal H_{0,h})=\{u\in H^2((0,h^{-\rho}))~:~u'(0)=-u(0)\quad{\rm and}\quad u(h^{-\rho})=0\}\,.
\end{equation}
The spectrum of the operator $\mathcal H_{0,h}$ is purely discrete
and consists of a strictly increasing sequence of eigenvalues
denoted by $(\lambda_n(\mathcal H_{0,h}))$.\\
 Note that this operator is associated with the quadratic form
$$
V_h\ni u\mapsto \int_0^{h^{-\rho}} |u'(\tau)|^2\,d\tau\,  - |u (0)|^2\,,
$$
with $V_h :=\{v\in H^1((0,h^{-\rho}))\,|\, v(h^{-\rho})=0\}$.

The next lemma gives the  localization of  the two first eigenvalues
$\lambda_1(\mathcal H_{0,h})$ and $\lambda_2(\mathcal H_{0,h})$ for
small values of $h$.

\begin{lem}\label{lem:1DL}
As $h\to0_+$, there holds
\begin{equation}\label{eq:lwh}
\lambda_1(\mathcal H_{0,h})= - 1 + 4 \big(1+o(1)\big) \exp\big( - 2 h^{-\rho}\big)\quad{\rm
and}\quad \lambda_2(\mathcal H_{0,h})\geq 0\,.
\end{equation}
\end{lem}
\begin{proof}
Let $w\geq0$ and $\lambda=-w^2$ be a negative eigenvalue of the
operator $\mathcal H_{0,h}$ with an eigenfunction $u$. We have,
\begin{equation} \label{ab1}
-u''=\lambda u\quad{\rm in~}(0,h^{-\rho})\,,\quad
u'(0)=-u(0)\,,\quad u(h^{-\rho})=0\,.
\end{equation}
 If $w=0$ and $h<1$, then the unique
 solution of \eqref{ab1} is $u=0$. Thus, we suppose that $w>0$. Easy
 computations give us
\begin{equation} \label{ab2}
u(\tau)=A\Big(\exp(-w\tau)-\exp(-2wh^{-\rho})\exp(w\tau)\Big)\,,
\end{equation}
with the following condition
on $w$,
\begin{equation} \label{ab3}
w\Big(1+\exp(-2wh^{-\rho})\Big)=1-\exp(-2wh^{-\rho})\,.
\end{equation}
Define the function
$$f(v)=v\Big(1+\exp(-2vh^{-\rho})\Big)-\Big(1-\exp(-2vh^{-\rho})\Big)=v-1+(v+1)\exp(-2vh^{-\rho})\,.$$
We will locate the zeros of $f$ in $\R_+=(0,\infty)$. Notice that
$f(v)> 0$ for all $v\geq 1$. This forces the solution $w$ of
\eqref{ab3} to live in $(0,1)$. 

{ 
If $v\in[h^{\rho/2},1-h^{\rho/2}]$, then as $h\to0_+$, the term $\exp(-2vh^{-\rho})$ is exponentially small and
$$\frac{v+1}{v-1}\exp(-2vh^{-\rho})\to0\,.$$
Writing 
$$f(v)=(v-1)\left(1+\frac{v+1}{v-1}\exp(-2vh^{-\rho})\right)\,,$$
}
we observe that  there exists
$h_0\in(0,1)$ such that, for all $h\in(0,h_0)$,
$$
f(v)<0 
\quad{\rm in}\quad [h^{\rho/2},1-h^{\rho/2}]\,.$$
This forces $w$ to live in $(0,h^{\rho/2})\cup (1-h^{\rho/2},1)$.
Since $f(v)=-1+o(1)$ as $v\to0_+$, then there exists $h_0\in(0,1)$
such that, for all $h\in(0,h_0)$ and $v\in (0,h^{\rho/2})$,
$f(v)<0$. Thus, $f$ may vanish  in $(1-h^{\rho/2},1)$ only. Notice
that
$$
f'(v)=1-2h^{-\rho}\left(v+1-\frac{h^\rho}2\right)\exp(-2vh^{-\rho})>0\quad{\rm
in}\quad(1-h^{\rho/2},1)\,.$$ 
Thus $f$ has a unique zero $w_h$ in
the interval $(1-h^{\rho/2},1)$. This zero corresponds to the first
eigenvalue of $\mathcal H_{0,h}$ by the relation $\lambda_1(\mathcal
H_{0,h})=-w_h^2$. Inserting
$$w_h=1+\mathcal O(h^{\rho/2})$$ into \eqref{ab3}, we get that
\begin{equation}\label{eq:wh}
w_h= 1 -  2 \big(1+ o(1)\big) \exp\big( - 2 h^{-\rho}\big)\,.
\end{equation}  As a consequence, we observe that
$\lambda_1(\mathcal H_{0,h})=-w_h^2$ satisfies \eqref{eq:lwh}.
Since $f$ does not vanish in $\R_+\setminus (1-h^{\rho/2},1)$, then
$\lambda_2(\mathcal H_{0,h})\geq 0\,.$
\end{proof}

\begin{rem} Note that by domain monotonicity and minimax
characterization, we get directly { from} the spectrum of $\mathcal
H_{0,0}$ that $\lambda_1(\mathcal H_{0,h} ) \geq -1$ and
$\lambda_2(\mathcal H_{0,h}) \geq 0$.
\end{rem}

\begin{rem}\label{rem:es}
Along the proof of Lemma~\ref{lem:1DL}, we obtain that the
eigenspace of the first eigenvalue $\lambda_1(\mathcal H_{0,h})$ is
generated by the function
\begin{equation}\label{u0h}
u_{0,h}(\tau)=A_h\Big(\exp(-w_h\tau)-\exp(-2w_hh^{-\rho})\exp(w_h\tau)\Big)\,,
\end{equation}
with $w_h$ a constant satisfying \eqref{eq:wh}. 
The constant $A_h$ is selected { so that} $u_{0,h}$ is normalized in
$L^2((0,h^{-\rho}))$. As $h\to0_+$, $A_h$ satisfies,
\begin{equation}\label{A_h}
A_h=\sqrt{2}+\mathcal O\Big(h^{-\rho/2}\exp(-h^{-\rho})\Big)\,.
\end{equation}
\end{rem}

\subsection{1D operator in a weighted space}~\\
Let $h\in(0,1)$, $\beta\in\R$, $\rho\in(0,1/2)$ and
$|\beta|h^{\frac12-\rho}<\frac13$. Consider the self-adjoint
operator
\begin{equation}\label{eq:H0b}
\mathcal H_{\beta,h}=-\frac{d^2}{d\tau^2}+\beta h^{1/2}(1-\beta
h^{1/2}\tau)^{-1}\frac{d}{d\tau}\quad{\rm in~}\quad
L^2\big((0,h^{-\rho});(1-\beta h^{1/2}\tau)d\tau\big)\,,\end{equation} with domain
\begin{equation}\label{eq:domH0b}
D(\mathcal H_{\beta,h})=\{u\in H^2((0,h^{-\rho}))~:~u'(0)=-u(0)\quad{\rm and}\quad u(h^{-\rho})=0\}\,.
\end{equation}
The operator $\mathcal H_{\beta,h}$ is the Friedrichs extension in
$L^2\big((0,h^{-\rho});(1-\beta h^{1/2}\tau)d\tau\big)$  associated
with the quadratic form defined for $u\in V_h$, by
$$q_{\beta,h}(u)=\int_0^{h^{-\rho}}|u'(\tau)|^2(1-\beta
h^{1/2}\tau)\,d\tau-|u(0)|^2\,.$$ The operator $\mathcal
H_{\beta,h}$ is with compact resolvent. The strictly increasing
sequence of the eigenvalues of $\mathcal H_{\beta,h}$ is denoted by
$(\lambda_n(\mathcal H_{\beta,h}))_{n\in \mathbb N^*}$.

The next lemma  localizes the spectrum of $\mathcal H_{\beta,h}$
near that of the operator $\mathcal H_{0,h}$ as $h$ goes to $0$.

\begin{lem}\label{lem:H0b}
There exist constants $C>0$ and $h_0\in(0,1)$ such that, for all
$h\in(0,h_0)$ and $n\in\mathbb N^*$, there holds,
$$\left|\lambda_n(\mathcal H_{\beta,h})-\lambda_n(\mathcal
H_{0,h})\right|\leq C|\beta|h^{\frac12-\rho}\,\big|\lambda_n(\mathcal H_{0,h})\big|\,.$$
\end{lem}
\begin{proof}
There exists a constant $C>0$ such that, for all $u\in
H^1((0,h^{-\rho}))$,
\begin{align*}
&\Big|q_{\beta,h}(u)-q_{0,h}(u)\Big|\leq |\beta|h^{\frac12-\rho}\,q_{0,h}(u)\,,\\
&\Big|\|u\|^2_{L^2((0,h^{-\rho});(1-\beta h^{1/2}\tau)d\tau)}-\|u\|^2_{L^2((0,h^{-\rho});d\tau)}\Big|\leq |\beta|h^{\frac12-\rho}\,\|u\|^2_{L^2((0,h^{-\rho});d\tau)}\,.
\end{align*}
The conclusion of the lemma is now a simple application of the
min-max principle.
\end{proof}

The next proposition states a two-term asymptotic expansion of the
eigenvalue $\lambda_1(\mathcal H_{\beta,h})$ as $h \to 0_+$.

\begin{pro}\label{lem:H0b;l}
There exist constants $C>0$ and $h_0\in(0,1)$ such that, for all
$h\in(0,h_0)$ and $|\beta|h^\rho<\frac13$, there holds,
$$\Big|\lambda_1(\mathcal H_{\beta,h})-(-1-\beta h^{1/2})\Big|\leq
C\beta^2h\,.$$
\end{pro}
\begin{proof}
Consider the function
$$f(\tau)=\chi(\tau\, h^{\rho})\,u_0(\tau)\,,$$
where
$$u_0(\tau)=\sqrt{2}\,\exp(-\tau)\,,$$
and $\chi\in C_c^\infty([0,\infty))$ satisfies
$$0\leq \chi\leq 1~{\rm in~}[0,\infty)\,,\quad \chi=1~{\rm
in~}[0,1/2)\quad{\rm and}\quad\chi=0~{\rm in~}[1/2,\infty)\,.$$
Clearly, the function $f$ is in the domain of the operator $\mathcal
H_{\beta,h}$. It is easy to get the estimates:
\begin{align*}
&\Big|\|f\|^2_{{L^2((0,h^{-\rho});(1-\beta
h^{1/2}\tau)d\tau)}}-1\Big|\leq C\beta^2h\,,\\
& \|\{\mathcal H_{\beta,h}-(-1-\beta
h^{1/2})\}f\|_{{L^2((0,h^{-\rho});(1-\beta h^{1/2}\tau)d\tau)}}\leq
C\beta^2h\,.
\end{align*}
By the spectral theorem, we deduce that there exists an eigenvalue
$\lambda(\mathcal H_{\beta,h})$ of $\mathcal H_{\beta,h}$ such that
$$\Big|\lambda(\mathcal H_{\beta,h})-(-1-\beta h^{1/2})\Big|\leq
C\,\beta ^2h\,.$$
Thanks to Lemmas~\ref{lem:1DL}~and~\ref{lem:H0b}, we  have
$$\lambda_1(\mathcal H_{\beta,h})=\lambda(\mathcal H_{\beta,h})\,.$$
\end{proof}


\section{Localization of ground states}

\subsection{Localization near the boundary}

\begin{thm}\label{thm:dec} Let $M\in(-1,0)$. For all $\alpha <
\sqrt{-M}$,  there exist constants $C>0$  and $h_0\in(0,1)$ such
that, if $u_h$ is a normalized eigenfunction of $\mathcal L_h$ with
eigenvalue $\mu(h)\leq Mh$, then, for all $h\in(0,h_0)$,
$$\int_\Omega \left(|u_h(x)|^2+h|\nabla u_h(x)|^2\right)\exp\left(\frac{2\alpha\, {\rm
dist}(x,\partial\Omega)}{h^{1/2}}\right)\,dx\leq C\,.$$
\end{thm}
\begin{proof}
Let $t(x)={\rm dist }(x,\partial\Omega)$ and
$\Phi(x)=\exp(\frac{\alpha\, t(x)}{h^{1/2}})$. We have the simple identity
$$
\langle \Phi^2\,u_h,\mathcal L_h u_h\rangle_{L^2(\Omega)}
=\mu(h)\|\Phi\,u_h\|^2_{L^2(\Omega)}\,.
$$
By an integration by parts, we get  the useful identity
\begin{equation}\label{eq:decomp}
q_h^\Phi(u_h):=\int_\Omega\left(|h\nabla(\Phi\,u_h)|^2-h^2|\nabla
\Phi|^2|u_h|^2\right)\,dx-h^{3/2}\int_{\partial\Omega}|\Phi\,u_h|^2\,ds(x)=\mu(h)\|\Phi\,u_h\|^2_{L^2(\Omega)}\,.\end{equation}
Consider a partition of unity of $\R$
$$\chi_1^2+\chi_2^2=1\,,$$
{ such that $\chi_1=1$ in $(-\infty,1)$, ${\rm supp
}\chi_1\subset (-\infty,2)$, $\chi_1\geq 0$ and $\chi_2\geq 0$ in $\R$. }

Define
$$\chi_{j,h}(x)=\chi_j\left(\frac{t(x)}{h^{1/2}}\right)\,,\quad
j\in\{1,2\}\,.$$ Associated with this partition of unity,  we have
the simple standard decomposition
$$q_h^\Phi(u_h)=\sum_{j=1}^2 q_{j,h}^\Phi(\,u_h)\,,$$
where
\begin{equation}\label{eq:decompa}
q_{1,h}^\Phi(u_h)=\int_\Omega\left(|h\nabla(\chi_{1,h}\,\Phi\,u_h)|^2-h^2|\nabla
(\chi_{1,h}\Phi)|^2|u_h|^2\right)\,dx-h^{3/2}\int_{\partial\Omega}|\chi_{1,h}\Phi\,u_h|^2\,ds(x)\,,
\end{equation}
and
\begin{equation}\label{eq:decompb}
 q_{2,h}^\Phi(u_h)=\int_\Omega  \left( |h\nabla(\chi_{2,h}\,\Phi\,u_h)|^2 - h^2|\nabla
(\chi_{2,h}\Phi)|^2|u_h|^2\right) \,dx\,.
\end{equation}
A rearrangement of the terms in \eqref{eq:decomp}-\eqref{eq:decompb} will yield the existence of a constant $C$ such that
\begin{multline}\label{eq:decomp1}
\int_{\{t(x)\leq
2h^{1/2}\}}\left(|h\nabla(\chi_{1,h}\,\Phi\,u_h)|^2-Ch|\Phi\,u_h|^2-\mu(h)|\chi_{1,h}\Phi\,u_h|^2\right)\,dx-h^{3/2}\int_{\partial\Omega}|\chi_{1,h}\Phi\,u_h|^2\,ds(x)
\\
\leq \int_\Omega\left(-|h\nabla(\chi_{2,h}\,\Phi\,u_h)|^2+(M
+\alpha^2)h |\chi_{2,h}\,\Phi\,u_h|^2 \right)\,dx\,.\end{multline}
The definition
of $\Phi$, the assumption on $\mu(h)<0$, and the normalization of
$u_h$ yield
$$-\int_{\{t(x)\leq
2h^{1/2}\}}\mu(h)|\chi_{1,h} \Phi\,u_h|^2\,dx\geq0\,,$$ 
and { (notice that $0\leq \Phi(x)\leq \exp(2\alpha)$ when $t(x)\leq 2h^{1/2}$)}
$$ -\int_{\{t(x)\leq
2h^{1/2}\}}h|\Phi\,u_h|^2\geq -2C h \,.$$  At this stage we are
left with the inequality:
\begin{multline}\label{eq:decomp1bis}
\int_{\{t(x)\leq
2h^{1/2}\}}\left(|h\nabla(\chi_{1,h}\,\Phi\,u_h)|^2 \right)\,dx-h^{3/2}\int_{\partial\Omega}|\chi_{1,h}\Phi\,u_h|^2\,ds(x)
\\
\leq \int_\Omega\left(-|h\nabla(\chi_{2,h}\,\Phi\,u_h)|^2+(M+
{ \alpha^2}) h |\chi_{2,h}\,\Phi\,u_h|^2  \right)\,dx\, + 2 C
h\,.\end{multline} 
In a small tubular neighborhood of the boundary, we may use the
boundary coordinates $(s,t)$ recalled in Section~\ref{sec:app} and
obtain
\begin{multline}\label{eq:rev-scaling} 
\int_{\{t(x)\leq
2h^{1/2}\}}|h\nabla(\chi_{1,h}\,\Phi\,u_h)|^2\,dx-h^{3/2}\int_{\partial\Omega}|\chi_{1,h}\Phi\,u_h|^2\,ds(x)\\
\geq
(1-Ch^{1/2})\int_{-|\partial\Omega|/2}^{|\partial\Omega|/2}\left(\int_0^{2h^{1/2}}h^2|\partial_t(\chi_{1,h}\,\Phi\,u_h)|^2\,dt-h^{3/2}|(\chi_{1,h}\,\Phi\,u_h)(s,t=0)|^2\right)\,
ds\,.
\end{multline}
{ 
Let $\frac12<\rho<1$.  We do the change of variable $\tau=h^{-1}t$ and introduce the function $v(s,\tau)=(\chi_{1,h}\,\Phi\,u_h)(s,t)$. The function $v$ vanishes for $\tau\geq 2h^{-1/2}$, and when $h$ is sufficiently small, $2h^{-1/2}\leq h^{-\rho}$. In this way we get
$$\int_0^{2h^{1/2}}h^2|\partial_t(\chi_{1,h}\,\Phi\,u_h)|^2\,dt-h^{3/2}|(\chi_{1,h}\,\Phi\,u_h)(s,t=0)|^2
=h\left(\int_0^{h^{-\rho}}|\partial_\tau v|^2\,d\tau-h^{1/2}|v(s,\tau=0)|^2\right)\,.$$
Now, we use the min-max principle and Lemma~\ref{lem:1DL} to obtain
$$\int_0^{h^{-\rho}}|\partial_\tau v|^2\,d\tau-h^{1/2}|v(s,\tau=0)|^2\geq 
\lambda_1(\mathcal H_{0,h})\int_0^{h^{-\rho}}|v|^2\,d\tau\geq -2h\int_0^{h^{-\rho}}|v|^2\,d\tau\,.$$
Coming back to the $t$ variable, we get
$$
\int_0^{2h^{1/2}}h^2|\partial_t(\chi_{1,h}\,\Phi\,u_h)|^2\,dt-h^{3/2}|(\chi_{1,h}\,\Phi\,u_h)(s,t=0)|^2\geq -2h\int_0^{2h^{1/2}}h^2|\partial_t(\chi_{1,h}\,\Phi\,u_h)|^2\,dt\,.
$$
Inserting this into \eqref{eq:rev-scaling}, we obtain}
\begin{multline*}
\int_{\{t(x)\leq
2h^{1/2}\}}|h\nabla(\chi_{1,h}\,\Phi\,u_h)|^2\,dx-h^{3/2}\int_{\partial\Omega}|\chi_{1,h}\Phi\,u_h|^2\,ds(x)
\\
\geq
-2(1-Ch^{1/2})h\int_{-|\partial\Omega|/2}^{|\partial\Omega|/2}\int_0^{2h^{1/2}}|\chi_{1,h}\,\Phi\,u_h|^2\,dtds\geq
-C'h\int_\Omega|\chi_{1,h}\,\Phi\,u_h|^2\,dx\geq -C'h\,.
\end{multline*}
Now, we infer from { \eqref{eq:decomp1bis}}
$$-(C'+2C)h
{ \geq} 
\int_\Omega\left(-|h\nabla(\chi_{2,h}\,\Phi\,u_h)|^2+ (M+ \alpha^2) h
|\chi_{2,h}\,\Phi\,u_h|^2\right)\,dx\,.$$ Since $(M + \alpha^2) <0$, we deduce that
$$\int_\Omega\left(|\nabla(\chi_{2,h}\,\Phi\,u_h)|^2 - (M + \alpha^2)
|\chi_{2,h}\,\Phi\,u_h|^2\right)\,dx\leq (C'+2C)\,,$$ which is enough to deduce the conclusion in Theorem~\ref{thm:dec}.
\end{proof}
~\\


\subsection{Localization near the point of maximal curvature}~\\

We will apply the result in Lemma~\ref{lem:H0b;l} to obtain a lower
bound of the quadratic form
\begin{equation}\label{eq:qf}
q_h(u)=\int|h\nabla u|^2\,dx-h^{3/2}\int_{\partial\Omega}|u|^2\,ds(x)\,.
\end{equation}
\begin{thm}\label{thm:lb-qf}
There exist constants $C>0$ and $h_0\in(0,1)$ such that, for all
$h\in(0,h_0)$ and $u\in H^1(\Omega)$,
$$q_h(u)\geq \int_\Omega U_h(x)|u(x)|^2\,dx\,,$$
where
$$U_h(x)=\left\{
\begin{array}{ll}
0&{\rm if~}{\rm dist}(x,\partial\Omega)> 2h^{1/8}\,,\\
 -h-\kappa(s(x))h^{3/2}-Ch^{7/4}&{\rm if~}{\rm dist}(x,\partial\Omega)\leq 2h^{1/8}\,.
\end{array}\right.$$
\end{thm}
\begin{proof}
Consider a partition of unity of $\R$,
$$
\chi_1^2+\chi_2^2=1
$$
with $\chi_1=1$ in $(-\infty,1]$ and ${\rm
supp}\chi_2\subset[1,\infty)$.
For $j\in\{1,2\}$, put,
$$
\chi_{j,h}(x)=\chi_j\left(\frac{{\rm
dist}(x,\partial\Omega)}{h^{1/8}}\right)\,.
$$
We have the  decomposition
$$q_h(u)=q_{1,h}(u)+q_{2,h}(u)-h^2\sum_{j=1}^2\big\|\,|\nabla\chi_{j,h}|u\,\big\|^2_{L^2(\Omega)}\,,$$
where
$$q_{1,h}(u)=\int_\Omega|h\nabla
(\chi_{1,h}\,u)|^2\,dx-h^{3/2}\int_{\partial\Omega}|u|^2\,ds(x)\,,$$
$$
q_{2,h}(u)=\int_\Omega|h\nabla
(\chi_{2,h}\,u)|^2\,dx\geq0\,,
$$
and
$${h^2} \big\|\,|\nabla\chi_{j,h}|u\,\big\|^2_{L^2(\Omega)}\leq
C_1h^{7/4}\int_{\{{\rm dist}(x,\partial\Omega)\leq 2h^{1/8}\}}|u|^2\,dx\,.$$
Thus,
$$
q_h(u)\geq q_{1,h}(\chi_{1,h}\,u)-C_1\,h^{7/4}\int_{\{{\rm dist}(x,\partial\Omega)\leq 2h^{1/8}\}}|u|^2\,dx\,.
$$
In the support of $\chi_{1,h}$, we may use the boundary coordinates
$(s,t)$ recalled in Section~\ref{sec:app} and write
\begin{multline}\label{eq:est;der-s}
q_{1,h}(u)\geq\int_{-|\partial\Omega|/2}^{|\partial\Omega|/2}\Big(\int_0^{2h^{1/8}}|h\partial_t
(\chi_{1,h}\,u)|^2\,(1-\kappa(s)\,t)dt-h^{3/2}|(\chi_{1,h}\,u)(s,t=0)|^2\Big)\, ds \\
+\int_{-|\partial\Omega|/2}^{|\partial\Omega|/2}|\partial_s
(\chi_{1,h}\,u)|^2\,(1-\kappa(s)\,t)^{-1}dtds\,.\end{multline} \\
Recall the operator in $\mathcal H_{\beta,h}$ in \eqref{eq:H0b}. By
a simple scaling argument {(similar to the one appearing after \eqref{eq:rev-scaling})} and the min-max principle, we have
\begin{multline}\label{eq:imp;est}
\int_0^{2h^{1/8}}|h\partial_t
(\chi_{1,h}\,u)|^2\,(1-\kappa(s)\,t)dt-h^{3/2}|(\chi_{1,h}\,u)(s,t=0)|^2\\
\geq h\lambda_1(\mathcal
H_{\kappa(s),h})\int_0^{h^{1/8}}|\chi_{1,h}\,u|^2\,(1-\kappa(s)\,t)dt\,.\end{multline}
Thanks to  Proposition~\ref{lem:H0b;l}, we deduce the following
lower bound,
\begin{equation}\label{eq:imp;est'}
q_{1,h}(u)\geq
\int_\Omega(-h-\kappa(s(x))h^{3/2}-C_2h^2)|\chi_{1,h}\,u|^2\,dx\,.
\end{equation}
Therefore, after collecting the  lower bounds for $q_{1,h}(u)$ and
$q_{2,h}(u)$,  we deduce that
\begin{align*}
q_h(u)&\geq
\int_\Omega(-h-\kappa(s(x))h^{3/2}-C_2h^2)|\chi_{1,h}\,u|^2\,dx-C_1h^{7/4}\int_{\{{\rm dist}(x,\partial\Omega)\leq 2h^{1/8}\}}|u|^2\,dx\\
&\geq \int_{\{{\rm dist}(x,\partial\Omega)\leq
2h^{1/8}\}}(-h-\kappa(s(x))h^{3/2}-C_1h^{7/4}-C_2h^2)|u|^2\,dx\,.
\end{align*}
To finish the proof of Theorem~\ref{thm:lb-qf}, we select the
constant $C=C_1+C_2$ in the definition of $U_h$.
\end{proof}

\begin{rem}\label{rem:partial-s}
Let $M>0$ and $u_{h}$ a normalized eigenfunction of $\mathcal L_h$
with  eigenvalue satisfying $\mu(h)\leq
-h-h^{3/2}\kappa_{\max}+Mh^{7/4}$. Using  the inequalities in
\eqref{eq:est;der-s}-\eqref{eq:imp;est} and the upper bound
satisfied by $\mu(h)$, we get the interesting estimate
\begin{equation}\label{eq:rem:partial-s}
\|\partial_s(\chi_{1,h}u_h)\|_2\leq Ch^{-1/8}\,.\end{equation}
 To see this,
notice that $q_{1,h}(u_h)\leq q_h(u_h)\leq \mu(h)\leq
-h-h^{3/2}\kappa_{\max}+Mh^{7/4}$. Now plugging \eqref{eq:imp;est'}
into \eqref{eq:est;der-s}, we get
$$\int_{-|\partial\Omega|/2}^{|\partial\Omega|/2}|h\partial_s
(\chi_{1,h}\,u)|^2\,(1-\kappa(s)\,t)^{-1}dtds\leq
C\big(\mu(h)-(-h-h^{3/2}\kappa_{\max}-C_2h^2)\big)\leq Ch^{7/8}\,.$$
In the support of $\chi_{1,h}$, we may write
$(1-t\kappa(s))^{-1}\geq 1$ and get the inequality in
\eqref{eq:rem:partial-s}.
\end{rem}

\begin{rem}\label{rem:mu1}
 The min-max principle and the conclusion in
Theorem~\ref{thm:lb-qf} yield a lower bound on the first eigenvalue
of the operator $\mathcal L_h$, namely, there exists $h_0>0$ such
that, for all $h\in(0,h_0)$,
$$\mu_1(h)\geq -h-\kappa_{\max}h^{3/2}-Ch^{7/4}\,.$$
\end{rem}

As a consequence of Theorem~\ref{thm:lb-qf}, we can prove, following the same steps as in \cite[Thm.~4.9]{FH}, that:
\begin{thm}\label{thm:dec-h}
Let $M>0$ and $\chi_1\in C_c^\infty(\overline{\R_+}\,)$. There exist
constants $C>0$, $\alpha>0$ and $h_0\in(0,1)$ such that, if $u_h$ is
a normalized eigenfunction of the operator $\mathcal L_h$ with
eigenfunction $\mu(h)\leq -h-\kappa_{\max}h^{3/2}+Mh^{7/4}$, then
for all $h\in(0,h_0)$,
$$\int_\Omega \chi_1\left(\frac{{\rm
dist}(x,\partial\Omega)}{h^{1/8}}\right)^2\left\{|u_h(x)|^2+h|\nabla u_h(x)|^2\right\}\exp\left(\frac{2\alpha\, |s(x)|^2}{h^{1/4}}\right)\,dx\leq C\,.$$
\end{thm}

The localization of the bound states in
Theorem~\ref{thm:dec-h} near the set
$$n(\partial\Omega)=\{x\in\partial\Omega~:~\kappa(s(x))=\kappa_{\max}\}$$
can be expressed using the Agmon distance $\hat
d_{\partial\Omega}(s,n(\partial\Omega))$ in $\partial\Omega$
associated with the metric $(\kappa_{\max}-\kappa(s))ds^2$. Let us
define
$$\hat d(x,n(\partial\Omega),h)= \hat d_{\partial\Omega}(s(x),n(\partial\Omega))\chi_1(d(x,\partial\Omega))+h^{-1/4}d(x,\partial\Omega)\,.$$
As in \cite[Theorem~8.3.4]{FH-b},
Theorem~\ref{thm:lb-qf} yields the following conclusion. Let $M>0$.
There exists $\delta>0$ and for all $\epsilon>0$, there exist
$h_\epsilon>0$ and $C_\epsilon>0$ such that, for all
$h\in(0,h_\epsilon]$, if $u_h$ is a normalized ground state as in
Theorem~\ref{thm:dec-h}, then,
$$\Big\|\exp\big(\delta h^{-1/4}\hat
d(x,n(\partial\Omega),h)\big)\,u_h\Big\|_{L^2(\Omega)}\leq
C_\epsilon\exp(\epsilon h^{-1/4})\,.$$ This estimate may have an
advantage over that in Theorem~\ref{thm:dec-h} when approximating
the bound states by  suitably constructed trial states.

A simple consequence of Theorems~\ref{thm:dec} and \ref{thm:dec-h}
is:
\begin{cor}\label{cor:dec}
Let $\chi\in C_c^\infty(\,\overline{\R_+}\,)$ such that ${\rm
supp}\,\chi\subset[0,t_0)$. Under the assumptions in
Theorem~\ref{thm:dec-h}, if $t_0$ is sufficiently small, then there
exists a constant $C$ such that,
$$\int_\Omega|s(x)|^k\,|\chi(t(x))|^2|u_h(x)|^2\,dx\leq
C\, h^{k/8}\,.$$
\end{cor}
Here we have also used the first localization.
\section{Construction of approximate eigenfunctions}\label{sec:ub}
In this section, we construct trial states $(\phi_n)$ in the domain
of the operator $\mathcal L_h$ such that, for every fixed
$n\in\mathbb N$, we have
\begin{equation}\label{eq:ef}
\left\|\mathcal L_h\phi_n-\left(-h-\kappa_{\max}h^{3/2}+(2n-1)\sqrt{ \frac{-\kappa''(0)}{2} }\,h^{7/4}\right)\phi_n\right\|_{L^2(\Omega)}\leq
Ch^{15/8}\|\phi_n\|_{L^2(\Omega)}\,.
\end{equation}
Notice that a direct consequence of \eqref{eq:ef} is:
\begin{equation}\label{eq:ef'}
\left\langle\mathcal L_h\phi_n-\left(-h-\kappa_{\max}h^{3/2}+(2n-1)\sqrt{ \frac{-\kappa''(0)}{2} }\,h^{7/4}\right)\phi_n\,,\,\phi_n\right\rangle_{L^2(\Omega)}
\leq Ch^{15/8}\|\phi_n\|^2_{L^2(\Omega)}\,.
\end{equation}
The trial states $(\phi_n)$ will be supported near the boundary of
$\Omega$, so that  the use of the boundary coordinates $(s,t)$
recalled in Section~\ref{sec:app} is valid. The operator $\mathcal
L_h$ is expressed in $(s,t)$ coordinates as follows:
$$\mathcal
L_h=-h^2a^{-1}\partial_s(a^{-1}\partial_s)-h^2a^{-1}\partial_t(a\partial_t)\quad \big({\rm in}~L^2(adsdt)\big)\,,$$
where
$$a(s,t)=1-t\kappa(s)\,.$$
In boundary coordinates, the Robin condition in \eqref{eq:bc}
becomes
$$h\partial_tu=-h^{3/2}u\quad{\rm on}\quad t=0\,.$$
The change of variable $(s,t)\mapsto (h^{1/8}\sigma,h^{1/2}\tau)$
transforms the above expression of $\mathcal L_h$ to
\begin{equation}\label{eq:scalingLh}
\mathcal L_h=h\left(-h^{3/4}\tilde
a^{-1}\partial_\sigma(\tilde a^{-1}\partial_\sigma)-\tilde
a^{-1}\partial_\tau(\tilde a\partial_\tau)\right)\,,\end{equation}
 where
\begin{equation}\label{eq:newa}
\tilde a(\sigma,\tau)=1-h^{1/2}\tau\kappa(h^{1/8}\sigma)\,.
\end{equation}
The boundary condition is now transformed to
$$\partial_\tau\widetilde u=-\widetilde u\quad{\rm on}\quad \tau=0\,.$$
  Recall that the
value of the maximal curvature is $\kappa_{\max}=\kappa(0)$. We have
the following asymptotic expansions:
\begin{align}
&\tilde a(\sigma,\tau)=1- h^{1/2} \tau \kappa (0) + h^{1/2}\tau\epsilon_1(\sigma)\,,\label{atilde}\\
& \tilde a^{-1}\partial_\tau\widetilde a(\sigma,\tau)= -h^{1/2}\kappa(0)-h^{3/4}\frac{\kappa''(0)}{2!}\sigma^2+h^{7/8}\epsilon_2(\sigma,\tau)\,,\\
&\tilde a(\sigma,\tau)^{-1}=1+h^{1/2}\epsilon_3(\sigma,\tau)\,,\label{atilde-1}\\
&\tilde a(\sigma,\tau)^{-2}=1+h^{1/2}\epsilon_4(\sigma,\tau)\,,\\
&\partial_\sigma\left\{\tilde a(\sigma,\tau)^{-1}\right\}=-h^{5/8}\epsilon_5(\sigma,\tau)\,,
\end{align}
where, for $0<h<\frac12$, $0\leq \tau\leq \mathcal O(h^{-\rho})$ and
$\sigma= \mathcal O(h^{-1/8})$, the functions $\epsilon_j$,
$j=1,\cdots,5,$ satisfy,
\begin{multline}\label{estabc}
 |\epsilon_1(\sigma)|\leq C|\sigma|^2\,,\quad
|\epsilon_2(\sigma,\tau)|\leq
C(|\sigma|^3+h^{1/8}\tau)\,,\quad|\epsilon_3(\sigma,\tau)|+
|\epsilon_4(\sigma,\tau)|+|\epsilon_5(\sigma,\tau)|\leq C\tau\,.
\end{multline}
 This gives us the following identities
\begin{align}
-\tilde a^{-1}\partial_\tau(\tilde
a\partial_\tau)&=-\partial_\tau^2+h^{1/2}\kappa(0)\partial_\tau+
h^{3/4}\frac{\kappa''(0)}{2!}\sigma^2\partial_\tau+h^{7/8}q_1(\sigma,\tau)\partial_\tau\,,\label{eq:dec-Lh2}\\
-h^{3/4}\tilde a^{-1}\partial_\sigma(\tilde
a^{-1}\partial_\sigma)&  =-h^{3/4}\partial_\sigma^2+h^{7/8}\left(h^{3/8}q_2(\sigma,\tau)\partial_\sigma^2+ h^{1/2}q_3(\sigma,\tau)\partial_\sigma\right)\,.\label{eq:dec-Lh1}
\end{align}
where the functions $q_1,q_2,q_3$ satisfy for $0<h<\frac12$, $0\leq
\tau\leq \mathcal O(h^{-\rho})$ and $|\sigma| =  \mathcal
O(h^{-1/8})$,
\begin{equation}\label{eq:qi}
 |q_1(\sigma,\tau)|\leq
C\, (|\sigma|^3+h^{1/8}\tau)\quad{\rm and}\quad
|q_2(\sigma,\tau)|+|q_3(\sigma,\tau)|\leq C\, \tau\,.
\end{equation}
We note for further use the  expressions of the functions
$q_1,q_2,q_3$:
\begin{multline}\label{eq:expq}
q_1(\sigma,\tau)=h^{-7/8}\left\{-\widetilde
a(\sigma,\tau)^{-1}\partial_\tau\big(\widetilde
a(\sigma,\tau)\big)-h^{1/2}\kappa(0)-h^{3/4}\frac{\kappa''(0)}{2}\sigma^2\right\}\,,\\
q_2(\sigma,\tau)=h^{-1/8}\big(1-\widetilde
a(\sigma,\tau)\big)\,,\quad q_3(\sigma,\tau)=-\widetilde
a(\sigma,\tau)\tau\kappa'(h^{1/8}\sigma)\,.
\end{multline}
Consequently, we have the formal expansion of the operator $\mathcal
L_h$,
\begin{equation}\label{eq:dec-Lh}
h^{-1}\mathcal
L_h=P_0+h^{1/2}P_2+h^{3/4}P_3+h^{7/8}Q_h\,,\end{equation} where
\begin{align}\label{defPj}
&P_0=-\partial_\tau^2\,, \nonumber\\
&P_2= \kappa(0)\partial_\tau= \kappa_{\max}\partial_\tau\,,\nonumber\\
&P_3=-\partial_\sigma^2+\frac{\kappa''(0)}{2}\sigma^2\partial_\tau\,,\nonumber\\
&Q_h=q_1(\sigma,\tau)\partial_\tau+h^{3/8}q_2(\sigma,\tau)\partial_\sigma^2+h^{3/8}q_3(\sigma,\tau)\partial_\sigma\,.
\end{align}
 Since
$\kappa(0)$ is the non-degenerate maximum of $\kappa$, then
$\kappa''(0)<0$. The eigenvalues of the harmonic oscillator
$$-\partial_\sigma^2+\frac{-\kappa''(0)}{2}\sigma^2\quad{\rm
in~}L^2(\R)$$ are $(2n-1)\sqrt{\frac{-\kappa''(0)}{2}}$ with
$n\in\mathbb N$. \\
The corresponding normalized eigenfunctions are denoted
by $f_n(\sigma)$. They have the form
\begin{equation}\label{fnsigma}
f_n(\sigma) = h_n(\sigma) \exp { \left(
- \sqrt{-\frac{\kappa''(0)}{2}}\frac{ \sigma^2}{2}
 \right)}
\,,
\end{equation}
where the $h_n(\sigma)$ are the rescaled Hermite polynomials.\\

Define the trial state $\phi_n\in L^2(\Omega)$ in $(s,t)$
coordinates as follows:
\begin{equation}\label{eq:ts}
\phi_n(s,t)=h^{-5/16}\,\chi_1\left(\frac{s}{|\partial\Omega|}\right)\,\chi_1\left(\frac{t}{h^{1/8}}\right)\,u_0(h^{-1/2}\,t)\,f_n(h^{-1/8}s)\,,
\end{equation}
where
\begin{equation}\label{defu0}
u_0(\tau)=\sqrt{2}\exp(-\tau)\,.
\end{equation}
The function $u_0$ should be interpreted as the ground state of the
model operator $\mathcal H_{0,0}$ introduced in \eqref{defH00}.

  The construction of the
trial state $\phi_n$ is based on the simple observation that the
function $g(\sigma,\tau)=u_0(\tau)f_n(\sigma)$ satisfies
$$P_0g=-g\,,\quad P_2g=-\kappa_{\max}g\,,\quad P_3g=(2n-1)g\,,$$
so that, in light of the expression of the operator $\mathcal L_h$
in \eqref{eq:dec-Lh}, we expect that \eqref{eq:ef} holds true. The
computations that we present below will  show that \eqref{eq:ef} is
indeed  true.

First, notice that, the explicit expression of $\phi_n$ and
\eqref{eq:bc;n} give us
\begin{equation}\label{eq:n-phi}
\|\phi_n\|_{L^2(\Omega)}=1+\mathcal O(h^{3/4})\,.
\end{equation}
Next, we observe that, after a change of variables,
\eqref{eq:dec-Lh} yields {(here $k_2=-\kappa''(0)$)}
\begin{equation}\label{eq:dec-Lh-ap}
h^{-1}\mathcal
L_h\phi_n(s,t)=\big(-1-\kappa_{\max}h^{1/2}+(2n+1)\frac{k_2}{2}h^{3/4}\big)\varphi_n(\sigma,\tau)+R\varphi_n(\sigma,\tau)\,,\end{equation}
where $\varphi_n(\sigma,\tau)=\phi_n(h^{1/8}\sigma,h^{1/2}\tau)$ and
\begin{align*}
h^{5/16}\,R\varphi_n(\sigma,\tau)&=-\kappa_{\max}h^{1/2}h^{3/8}\chi_1'(h^{3/8}\tau)\chi_1\left(\frac{h^{1/8}\sigma}{|\partial\Omega|}\right)\,u_0(\tau)\,f_n(\sigma)\\
&~~
-h^{3/8}\Big(2\chi_1'(h^{3/8}\tau)+h^{3/8}\chi_1''(h^{3/8}\tau)\Big)\chi_1\left(\frac{h^{1/8}\sigma}{|\partial\Omega|}\right)\,u_0(\tau)\,f_n(\sigma)\\
&~~-h^{7/8}\Big(2\chi_1'\left(\frac{h^{1/8}\sigma}{|\partial\Omega|}\right)+h^{1/8}\chi_1''\left(\frac{h^{1/8}\sigma}{|\partial\Omega|}\right)\Big)
\chi_1(h^{3/8}\tau)\,u_0(\tau)\,f_n(\sigma)\\
&~~+h^{7/8}Q_h\Big(\chi_1\left(h^{3/8}\tau\right)\,\chi_1\left(\frac{h^{1/8}\sigma}{|\partial\Omega|}\right)\,u_0(\tau)\,f_n(\sigma) \Big)\,.
\end{align*}
Using the expression of $Q_h$ and the exponential decay of $u_0$ and
$f_n$ at infinity, it is easy to check that
$$\int |R\varphi_n(\sigma,\tau)|^2\widetilde
a(\sigma,\tau)\,h^{5/8}d\sigma d\tau\leq Ch^{7/4}\,.$$ Now, we infer
from \eqref{eq:dec-Lh-ap}, \begin{multline*} \left\|\mathcal
L_h\phi_n-\big(-h-\kappa_{\max}h^{3/2}+(2n+1)\frac{k_2}{2}h^{7/4}\big)\phi_n\right\|_{L^2(\Omega)}\\
=h\left\|R\varphi_n(\sigma,\tau)\right\|_{L^2(h^{5/8}\widetilde
a(\sigma,\tau)d\sigma d\tau)}\leq Ch^{15/8}\,,\end{multline*}
thereby proving \eqref{eq:ef}.

Next, we indicate how \eqref{eq:ef} is useful to obtain upper bounds
of the eigenvalues of the operator $\mathcal L_h$. As a consequence
of the spectral theorem, \eqref{eq:ef} yields that,  for every fixed
$n\in\mathbb N$, there exists an eigenvalue $\widetilde\mu_n$ of the
operator $\mathcal L_h$ such that
$$\widetilde\mu_n=-h-\kappa_{\max}h^{3/2}+(2n-1)\sqrt{\frac{k_2}2}\,h^{7/4}+\mathcal
O(h^{15/8})\,.$$ Thus, for every $n\in\mathbb N$ and $h>0$
sufficiently small, we get an increasing sequence of distinct
eigenvalues $(\widetilde\mu_j)_{j=1}^n$ of the operator $\mathcal
L_h$. Clearly, the min-max eigenvalues $(\mu_j(\mathcal L_h))$ of
the operator $\mathcal L_h$ will satisfy
$$\forall~j\in\{1,2,\cdots,n\}\,,\quad\mu_j(\mathcal L_h)\leq
\max_{1\leq i\leq j}\widetilde\mu_i\leq \widetilde\mu_j.$$ In that way, we arrive at:
\begin{thm}\label{thm:st}
Let $n \in\mathbb N$. There exist constants $h_n\in(0,1)$ and
$C_n>0$ such that the min-max eigenvalues $(\mu_j(\mathcal
L_h))_{j=1}^n$ of the operator $\mathcal L_h$ satisfy, for all
$h\in(0,h_n)$,
$$\mu_j(\mathcal
L_h)\leq -h-\kappa_{\max}h^{3/2}+(2j-1)\sqrt{\frac{-\kappa''(0)}{2}}\,h^{7/4}+r_j(h)\,,$$
where
$$|r_j(h)|\leq C_n h^{15/8}\,.$$
\end{thm}
\section{Refined trial states}\label{sec:Gr}

\subsection{ Main Result} ~\\
In this section, we construct trial states that will give us,
together with the lower bounds that we will derive in the next
section, an accurate expansion of the eigenvalues of the operator
$\mathcal L_h$ valid to any order.

In \eqref{eq:dec-Lh}, we derived  the expression of the operator
$\mathcal L_h$ in the re-scaled boundary coordinates
$(\sigma,\tau)=(h^{-1/8}s,h^{-1/2}t)$. The expansion in
\eqref{eq:dec-Lh} involves the operators $P_0,P_2,P_3,Q_h$ in
\eqref{defPj}. The key to the construction of the trial states is to
expand the  operator $Q_h$ in powers of $h^{1/8}$.

The functions $q_1$, $q_2$ and $q_3$ in \eqref{defPj} admit Taylor
expansions that can be rearranged in the form:
\begin{equation}\label{eq:expqj}
\begin{aligned}
&q_1(\sigma,\tau)=\sum_{j=0}^\infty q_{1,j}(\sigma,\tau)h^{j/8}\,,\\
&h^{3/8}q_2(\sigma,\tau)=\sum_{j=0}^\infty q_{2,j}(\sigma,\tau)h^{j/8}\,,\\
&h^{3/8}q_3(\sigma,\tau)=\sum_{j=0}^\infty q_{3,j}(\sigma,\tau)h^{j/8}\,,
\end{aligned}
\end{equation}
with the functions $q_{k,j}(\sigma,\tau)$ being polynomial functions
of  $(\sigma,\tau)$.  The operator $Q_h$ in \eqref{defPj} admits the
formal expansion
\begin{equation}\label{expQh}
Q_h=\sum_{j=0}^\infty h^{j/8}Q_j\,,
\end{equation}
where
\begin{equation}\label{eq:Qj}
Q_j=q_{1,j}(\sigma,\tau)\partial_\tau+q_{2,j}(\sigma,\tau)\partial_\sigma^2+q_{3,j}(\sigma,\tau)\partial_\sigma\,.
\end{equation}
Furthermore,  for every $p\geq1$, $j\geq0$ and
$f\in\mathcal S(\R\times \overline{\R_+})$,  { $Q_jf\in L^p(\R\times\R_+)$}.

Remember from \eqref{defu0}  that $u_0(\tau)=\sqrt{2}\,\exp(-\tau)$ and $f_n(\sigma)$ is
the $n$-th normalized eigenfunction of the harmonic oscillator
$H_{\rm harm}=-\partial_\sigma^2+\frac{k_2}2\sigma^2$.

For all $M\geq0$, we introduce the two operators:
\begin{equation}\label{eq:LhM}
\widetilde{\mathcal L}_{h,M}=P_0+h^{1/2}P_2+h^{3/4}P_3+h^{7/8}F_M\quad{\rm and}\quad F_M=\sum_{j=0}^M h^{j/8}Q_j\,.
\end{equation}
Thanks to \eqref{eq:dec-Lh} and \eqref{expQh}, for every $f\in
\mathcal S(\R\times\R_+)$, there exists a constant $C>0$ such that,
for all $h\in(0,1)$,
\begin{equation}\label{eq:Lh-LhM}
\|h^{-1}\mathcal L_h
f-\widetilde{\mathcal L}_{h,M}f\|_{L^2(\R\times\R_+)}\leq Ch^{\frac78+\frac{M+1}8}\,.
\end{equation}

Let $n\in\mathbb N$. We will construct a sequence of real numbers
$(\zeta_{j,n})_{j=0}^\infty$ and two sequences of real-valued
Schwartz functions $(v_{{n,j}})_{j=1}^\infty\subset\mathcal S(\R)$,
$(g_{n,j})_{j=0}^\infty\subset\mathcal S(\R\times \overline{\R_+})$
such that
$$\forall~j,\quad \partial_\tau
g_{n,j}\Big|_{\tau=0}=-g_{n,j}\Big|_{\tau=0}\,,$$ and for all $M\in \mathbb N$,
the function
$$\Psi_{n,M}(\sigma,\tau)=u_0(\tau)f_n(\sigma)+\sum_{j=1}^{M+1}h^{j/8} u_0(\tau) v_{n,j}(\sigma)+h^{7/8}\sum_{j=0}^M h^{j/8} g_{n,j}(\sigma,\tau)$$
is in the Schwartz space $\mathcal S(\R\times\R_+)$ and satisfies
(for all $h\in(0,h_{n,M}))$
\begin{equation}\label{eq:formalef}
\Big\|\Big(\widetilde{\mathcal L}_{h,M}
-\mu_{h,M}\Big)\Psi_{n,M}\Big\|_{L^2(\R\times\R_+)}\leq
C_{n,M}\, h^{\frac78+\frac{M+1}8}\,,
\end{equation}
where
\begin{equation}\label{eq:mu-hM}
\mu_{n,M}=-1-h^{1/2}\kappa_{\max}+h^{3/4}\sqrt{\frac{k_2}2}\,(2n-1)+h^{7/8}\sum_{j=0}^M\zeta_{j,n}h^{j/8}\,,
\end{equation}
$C_{n,M}>0$ and $h_{n,M}$ are two  constants determined by the
values of $n$ and $M$ solely.

Let us mention that, for every $j$, our definition of the
function $g_{n,j}(\sigma,\tau)$ ensures that it is a finite sum of
functions  having the simple  form
$F(\sigma)\times U(\tau)$.

Before we proceed in the construction of the aforementioned
sequences, we show how they give us refined expansions of the
eigenvalues of the operator $\mathcal L_h$. Let $\chi_1$ be the
cut-off function as in \eqref{eq:ts}. Define a function $\phi_{n,M}$
in $L^2(\Omega)$ by means of the boundary coordinates as follows:
\begin{equation}\label{eq:tsM}
\phi_{n,M}(s,t)=\chi_1\left(\frac{s}{|\partial\Omega|}\right)\, \chi_1\left(\frac{t}{h^{1/8}}\right)\Psi_{n,M}(h^{-1/8}s,h^{-1/2}t)\,.
\end{equation}
The function $\phi_{n,M}$ is in the domain of the operator $\mathcal
L_h$ because the functions $g_{n,j}$ and $u_0$ satisfy the Robin
condition at $\tau=0$. Since all involved functions in the
expression of $\Psi_{n,M}$ are in the Schwartz space, 
multiplication  by  cut-off functions will produce small errors in the
various calculations--all the error terms are $\mathcal
O(h^\infty)$. Also, since the function $u_0(\tau)f_n(\sigma)$ is
normalized in $L^2(\R\times\R_+)$,  the norm of $\phi_{n,M}$ in
$L^2(\Omega)$ is equal to $1+o(1)$, as $h\to0_+$.

Thanks to \eqref{eq:Lh-LhM} and \eqref{eq:formalef}, we get that
$$\big\|\mathcal
L_h\phi_{n,M}-h\mu_{n,M}\phi_{n,M}\big\|_{L^2(\Omega)}=\mathcal
O\big(h^{\frac78+\frac{M+1}8}\big)\|\Psi_{n,M}\|_{L^2(\Omega)}\,.$$ The spectral theorem now gives us:

\begin{thm}\label{thm:ub}
Let $n\in\mathbb N$. For every  $M\geq 0$, there exists an
eigenvalue $\widetilde\mu_n(h)$ of the operator $\mathcal L_h$ such
that, as $h\to0_+$,
$$\widetilde\mu_n(h)=-h-h^{3/2}\kappa_{\max}+h^{3/4}\sqrt{\frac{k_2}2}\,(2n-1)+h^{7/8}\sum_{j=0}^M\zeta_{j,n}h^{j/8}+\mathcal
O\big(h^{\frac78+\frac{M+1}8}\big)\,.$$
\end{thm}

The rest of the section is  devoted to the construction of the
numbers $\zeta_{j,n}$ and the functions $\Psi_{n,M}$. That will be
done by an iteration process. For simplicity, we will drop $n$ from
the notation and write $\zeta_j,\Psi_j,\cdots$.

\subsection{The sequence $(\zeta_{j,n})$}

\subsubsection*{Construction of
$(\zeta_{n,0},\Psi_{n,0})=(\zeta_0,\Psi_0)$}~\\
We want to find a real number $\zeta_0$ and two functions
$v_1(\sigma)$, $g_0(\sigma,\tau)$  such that
\begin{equation*}
\Psi_0(\sigma,\tau)=u_0(\tau)f_n(\sigma)+h^{1/8}u_0(\tau)v_1(\sigma)+h^{7/8}g_0(\sigma,\tau)
\end{equation*}
and
\begin{equation*}
\mu_0=-1-h^{1/2}\kappa_{\max}+h^{3/4}\sqrt{\frac{k_2}2}\,(2n-1)
+\zeta_0h^{7/4}\end{equation*} satisfy \eqref{eq:Lh-LhM} for
$M=0\,$. \\
Recall that the expression of the operator
$\widetilde{\mathcal L}_{h,M}$ for $M=0$ { is}
$$\widetilde{\mathcal
L}_{h,0}=P_0+h^{1/2}P_2+h^{3/4}P_3+h^{7/8}F_0\,.$$ For every Schwartz function $v(\sigma)$, we have the two simple identities,
\begin{multline*}
(P_0+h^{1/2}P_2\big)u_0(\tau)v(\sigma)=\big(-1-h^{1/2}\kappa_{\max}\big)u_0(\tau)v(\sigma)\quad{\rm
and}\quad P_3u_0(\tau)v(\sigma)=u_0(\tau)H_{\rm harm}v(\sigma)\,,
\end{multline*}
where $H_{\rm harm}=-\frac{d^2}{d\sigma^2}+\frac{k_2}2\sigma^2$. A
straightforward computation yields
\begin{align}
&\Big(\widetilde{\mathcal L}_{h,0}-\mu_0\Big)\Psi_0\nonumber\\
&=h^{7/8}\Big[(P_0+1)g_0(\sigma,\tau)+u_0(\tau)\big(H_{\rm harm}-\sqrt{\frac{k_2}2}\,(2n-1)\big)v_1(\sigma)+(F_0-\zeta_0)u_0(\tau)f_n(\sigma)\Big]\nonumber\\
&~+h^{7/8}\Big[\big(h^{1/2}P_2+h^{3/4}P_3+h^{7/8}F_0+h^{1/2}\kappa_{\max}-h^{3/4}\sqrt{\frac{k_2}2}\,(2n-1)-\zeta_0h^{7/4}\big)g_0(\sigma,\tau)\Big]\nonumber\\
&~+h\Big[F_0-\zeta_0\Big]u_0(\tau)v_1(\sigma)\,.\label{eq:formalef+}
\end{align}
In order that \eqref{eq:Lh-LhM} is satisfied for $M=0$, it is
sufficient to select $\zeta_0$, $g_0(\sigma,\tau)$ and $v_1(\sigma)$
 as follows
\begin{equation}\label{eq:leadingT}
(P_0+1)g_0(\sigma,\tau)+u_0(\tau)\big(H_{\rm
harm}-\sqrt{\frac{k_2}2}\,(2n-1)\big)v_1(\sigma)+(F_0-\zeta_0)u_0(\tau)f_n(\sigma)=0\,.
\end{equation}

We can select $\zeta_0$ such that
$$(F_0-\zeta_0)u_0(\tau)f_n(\varsigma)
\perp u_0(\tau)f_n(\sigma) \quad{\rm in}~L^2(\R\times\R_+)\,.$$ This
is given by the formula:
\begin{equation}\label{eq:zeta0}
\zeta_0 = \langle F_0 (u_0\otimes f_n)\,,\, u_0\otimes f_n\rangle_{L^2(\R\times\R_+)}\,.
\end{equation}
%
Consequently, 
$$h_1(\sigma):=-\int_0^\infty u_0(\tau)(
F_0-\zeta_0)u_0(\tau)f_n(\sigma)\,d\tau\perp f_n(\sigma)\quad{\rm
in~}L^2(\R)\,.$$ Since $H_{\rm
harm}-\sqrt{\frac{k_2}2}\,(2n-1)$ can be inverted in the orthogonal
complement of the $n$-th eigenfunction $f_n(\sigma)$,  we may select
$v_1(\sigma)$ to be
$$v_1(\sigma)=\big(H_{\rm
harm}-\sqrt{\frac{k_2}2}\,(2n-1)\big)^{-1}(h_1(\sigma))\,.$$
%
As a
consequence of the choice of $\zeta_0$ and $v_1(\sigma)$, we get
that, for all $\sigma$,
\begin{multline*}
w_1(\sigma,\cdot):=-u_0(\cdot)\big(H_{\rm
harm}-\sqrt{\frac{k_2}2}\,(2n-1)\big)v_1(\sigma)-(F_0-\zeta_0)u_0(\cdot)f_n(\sigma)\\
=-u_0(\cdot)h_1(\sigma)-(F_0-\zeta_0)u_0(\cdot)f_n(\sigma)\end{multline*}
is in the orthogonal complement of  $\{u_0(\tau)\}$ in $L^2(\R_+)$,
i.e.
$$\int_0^\infty w_1(\sigma,\tau) u_0(\tau)\,d\tau=0\,.$$
 Furthermore, since
$F_0=q_{1,0}\partial_\tau+q_{2,0}\partial_\sigma^2+q_{3,0}\partial_\sigma$
and every $q_{i,0}$ is a polynomial in $(\sigma,\tau)$, then the
function $w_1(\sigma,\tau)$ can be expressed in the form
\begin{equation}\label{eq:decw1}
w_1(\sigma,\tau)=-h_1(\sigma)u_0(\tau)+G_1(\sigma)u_1(\tau)+G_2(\sigma)u_2(\tau)+G_3(\sigma)u_3(\tau)\,,
\end{equation}
where the functions $h_1$, $u_i$ and $G_i$ are Schwartz functions
(they inherit this property from the functions $u_0$ and $f_n$).

The operator $P_0+1$ can be inverted in the orthogonal complement of
 $\{u_0(\tau)\}$, and the inverse is an operator in
$L^2(\R_+)$  which sends $\mathcal S (\overline{\mathbb R_+})$ into
itself.  The proof of this is standard and follows the same
argument as in \cite[Lemma~A.5]{FH}. In that way, we may select
$g_0(\sigma,\tau)$ to be
$$g_0(\sigma,\cdot)=(P_0+1)^{-1}\big(w_1(\sigma,\cdot)\big)\,,$$
and this function is in the Schwarz space $\mathcal
S(\R\times\overline\R_+)$, thanks to the decomposition of $w_1$ in
\eqref{eq:decw1}.

{ Now, for $M=1$, the estimate in \eqref{eq:formalef} is verified. Furthermore, returning to \eqref{eq:formalef+} we have proved that:
\begin{equation}\label{eq:formalef++}
\Big(\widetilde{\mathcal L}_{h,0}-\mu_0\Big)\Psi_0=h\sum_{j=0}^{10}h^{j/8}r_{0,j}(\sigma,\tau)\,,
\end{equation}
where $r_{0,j}(\sigma,\tau)$ are Schwartz functions.}
\subsubsection*{The iteration process}~\\
{ Here we suppose that we have selected
$(\zeta_0,\cdots,\zeta_M;\Psi_0,\cdots\Psi_M)$ such that \eqref{eq:Lh-LhM} and the following induction hypothesis are satisfied:
\begin{equation}\label{eq:formalef+++}
\Big(\widetilde{\mathcal L}_{h,M}-\mu_M\Big)\Psi_M=h^{\frac78+\frac{M+1}8}\sum_{j=0}^{N_M}h^{j/8}r_{M,j}(\sigma,\tau)\,,
\end{equation}
where $r_{M,j}(\sigma,\tau)$ are Schwartz functions.
Thanks to \eqref{eq:formalef++}, the condition in \eqref{eq:formalef+++} holds for $M=0$.
We have to select $\zeta_{M+1}$
and two functions $v_{M+2}(\sigma)$, $g_{M+1}(\sigma,\tau)$ such
that,
$$\Psi_{M+1}=\Psi_M+h^{\frac{M+2}8}u_0(\tau)v_{M+1}(\sigma)+h^{\frac78+\frac{M+1}8}g_{M+1}(\sigma,\tau)\quad{\rm
and}\quad \mu_{M+1}=\mu_M+\zeta_{M+1}h^{\frac78+\frac{M+1}8}$$ satisfy \eqref{eq:Lh-LhM} and \eqref{eq:formalef+++}  for $M$ replaced by $M+1$ (and some collection of Schwartz functions $r_{M+1,j}(\sigma,\tau)$).}

Notice that we have
$$\widetilde{\mathcal L}_{h,M+1}=\widetilde{\mathcal
L}_{h,M}+h^{\frac78+\frac{M+1}8}Q_{M+1}\,,$$
%
and
\begin{align}
&\Big(\widetilde{\mathcal L}_{h,M+1}-\mu_{M+1}\Big)\Psi_{M+1}\nonumber\\
&=\Big(\widetilde{\mathcal L}_{h,M}-\mu_{M}\Big)\Psi_{M}\nonumber\\
&~+h^{\frac78+\frac{M+1}8}\Big[(P_0+1)g_{M+1}(\sigma,\tau)+u_0(\tau)\big(H_{\rm harm}-\sqrt{\frac{k_2}2}\,(2n-1)\big)v_{M+2}(\sigma)+(Q_{M+1}-\zeta_{M+1})u_0(\tau)f_n(\sigma)\Big]\nonumber\\
&~+h^{\frac78+\frac{M+1}8}\Big[\big(Q_{M+1}-\zeta_{M+1}\big)\big(\Psi_M-u_0(\tau)f_n(\sigma)\big)\Big]\nonumber\\
&~+h^{\frac78+\frac{M+1}8}\Big[
\big(h^{1/2}P_2+h^{3/4}P_3+h^{7/8}F_{M+1}-(\mu_M+1)\big)g_{M+1}(\sigma,\tau)\Big]\nonumber\\
&~+h^{\frac78+\frac{M+1}8}\Big[h^{1/8}F_{M+1}-\big(\mu_{M+1}+1+h^{1/2}\kappa_{\max}-h^{3/4}\sqrt{\frac{k_2}2}\,(2n-1)\big)\Big]u_0(\tau)v_{M+2}(\sigma)\,.\label{eq:formalef+'}
\end{align}
It is sufficient to select $(\zeta_{M+1},v_{M+2},g_{M+1})$ such that
\begin{multline*}
(P_0+1)g_{M+1}(\sigma,\tau)+u_0(\tau)\big(H_{\rm
harm}-\sqrt{\frac{k_2}2}\,(2n-1)\big)v_{M+2}(\sigma)\\
+(Q_{M+1}-\zeta_{M+1})u_0(\tau)f_n(\sigma){ +r_{M,0}(\sigma,\tau)}=0\,.\end{multline*}
We select $\zeta_{M+1}$ such that
$$(Q_{M+1}-\zeta_{M+1})u_0(\tau)f_n(\sigma)+r_{M,0}(\sigma,\tau)
\perp u_0(\tau)f_n(\sigma) \quad{\rm
in}~L^2(\R\times\R_+)\,,$$
i.e.
\begin{equation}\label{eq:zeta}
\zeta_{M+1}=\langle Q_{M+1}(u_0\otimes f_n),u_0\otimes f_n\rangle_{L^2(\R\times\R_+)}+\langle r_{M,0},u_0\otimes f_n\rangle_{L^2(\R\times\R_+)}\,.
\end{equation}
Consequently, 
$$h_{M+2}(\sigma):=-\int_0^\infty u_0(\tau)\Big(
(Q_{M+1}-\zeta_{M+1})u_0(\tau)f_n(\sigma)+{ r_{M,0}(\sigma,\tau)}\Big)\,d\tau\perp
f_n(\sigma)\quad{\rm in~}L^2(\R)\,.$$ We  select $v_{M+2}(\sigma)$ as follows
$$v_{M+2}(\sigma)=\big(H_{\rm
harm}-\sqrt{\frac{k_2}2}\,(2n-1)\big)^{-1}(h_{M+2}(\sigma))\,.$$ As
a consequence of the choice of $\zeta_{M+1}$ and $v_{M+2}(\sigma)$,
we get that, for all $\sigma$,
$$
w_{M+2}(\sigma,\cdot):=-u_0(\cdot)\big(H_{\rm
harm}-\sqrt{\frac{k_2}2}\,(2n-1)\big)v_{M+2}(\sigma)-(Q_{M+1}-\zeta_{M+1})u_0(\tau)f_n(\sigma)-r_{M,0}(\sigma,\tau)
\perp u_0(\cdot)$$  in $L^2(\R_+)$.\\
 Finally, we select
$g_{M+1}(\sigma,\tau)$ as follows
$$g_{M+1}(\sigma,\cdot)=(P_0+1)^{-1}\big(w_{M+2}(\sigma,\cdot)\big)\,.$$
{ Having selected $\zeta_{M+1}$, $v_{M+2}$ and $g_{M+1}$, the expansions in \eqref{eq:formalef+'} and \eqref{eq:formalef+++} yield that
$$
\Big(\widetilde{\mathcal L}_{h,M+1}-\mu_M\Big)\Psi_{M+1}=h^{\frac78+\frac{M+2}8}\sum_{j=0}^{N_{M+1}}h^{j/8}r_{M+1,j}(\sigma,\tau)\,,
$$
where $r_{M+1,j}(\sigma,\tau)$ are Schwartz functions.}

\subsection{Vanishing of even indexed $\zeta_{j,n}$}\label{sec:zeta} ~\\
Let us start by observing that $f_1(\sigma)$ is a Gaussian, hence an
even function. The other eigenfunctions $f_n(\sigma)$ are generated
from $f_1(\sigma)$ by the recursive relation
$$f_n(\sigma)=c_n(L^+)^{n-1}f_1(\sigma)\,,$$
where the constant $c_n$ normalizes $f_n$ in $L^2(\R)$, and
$L^+=-\frac{d}{d\sigma}+\sigma$. The operator $L^+$ transforms even
functions to odd functions, and odd functions to even functions. In
that way, we observe that, $f_n(\sigma)$ is even when $n$ is odd,
and odd otherwise.

Next we examine carefully the terms $q_k(\sigma,\tau)$,
$k\in\{1,2,3\}$. Thanks to \eqref{eq:expq}, we see that,
$$
q_1(\sigma,\tau)=h^{-3/8}\alpha(h^{1/8}\sigma,h^{1/2}\tau)\,,$$ with
$\alpha(x,y)$ being a smooth function. We will abuse notation and
identify $q_1$ with its Taylor expansion. In that way we write,
\begin{multline*}
q_1(\sigma,\tau)=h^{-3/8}\sum_{m=3}^\infty\sum_{k=0}^m
\frac{\partial^m\alpha}{\partial x^k\partial
y^{m-k}}(0,0)\,(h^{1/8})^k(h^{1/2})^{m-k}\sigma^k\tau^{m-k}\\=\sum_{m=3}^\infty\sum_{k=0}^m
\frac{\partial^m\alpha}{\partial x^k\partial
y^{m-k}}(0,0)\,h^{\frac{k-3}8}\,h^{\frac{4(m-k)}8}\,\sigma^k\tau^{m-k}\,.\end{multline*}
Consequently, in the above expansion of $q_1(\sigma,\tau)$, we see
that the coefficients of the terms with odd powers of $\sigma$ have
even powers of $h^{1/8}$ (actually $k$ is odd iff $(k-3)+4(m-k)$ is
even). Returning to the expansion of $q_1(\sigma,\tau)$ in
\eqref{eq:expqj}, we get that, for all even $j$, the function
$q_{1,j}(\sigma,\tau)$ is an odd function of $\sigma$, while for odd
$j$, it is an even function of $\sigma$.

A similar analysis of the terms $q_2(\sigma,\tau)$ and
$q_3(\sigma,\tau)$ shows that, for all $j$, the term
$q_{2,j}(\sigma,\tau)$ is an odd function of $\sigma$ iff $j$ is
even, and $q_{3,j}(\sigma,\tau)$ is an odd function of $\sigma$ iff
$j$ is odd.

By the discussion on the parity of the function $f_n(\sigma)$, we
see that, if $M$ is even,
$$Q_M u_0(\tau) f_n(\sigma)=u_0(\tau)\Big(
-q_{1,M}f_n(\sigma)+q_{2,M}f_n''(\sigma)+q_{3,j}f'_n(\sigma)\Big)$$
is an odd function of $\sigma$ when $n$ is odd, and  it is an even
function  when $n$ is even. Consequently, for every $n$, if $M$ is
even, then the function
$$\Big(Q_M u_0(\tau) f_n(\sigma)\Big)
u_0(\tau)f_n(\sigma)$$ is an odd function of $\sigma$. \\
Consequently,
 for all values of $\tau$,
$$\int_\R \Big(Q_M u_0(\tau) f_n(\sigma)\Big)
u_0(\tau)f_n(\sigma)\,d\sigma=0\,.$$
{ Similarly, one can follow by recursion the parity of $r_{M,0}$ with respect to $\sigma$ and get for $M$ even
$$\int_\R r_{M,0}(\sigma,\tau)
u_0(\tau)f_n(\sigma)\,d\sigma=0\,.$$
Inserting this into \eqref{eq:zeta0} and \eqref{eq:zeta}, we get that $\zeta_{M,n}=0$ for every odd $M$.\\
}

\section{Lower bounds for the low-lying eigenvalues}\label{sec:lb}
\subsection{Useful operators}
Let $\rho\in(0,1)$. The number $\rho$ will be selected as
$\rho=\frac18+\eta$ with $\eta\in(0,\frac1{24})$.  In the
sequel, $\sigma$ and $\tau$ will denoted two variables. The variable
$\sigma$ will live in $\R$ (or in the interval
$(h^{-\eta},h^{\eta})$. The variable $\tau$ will live in the
interval $(0,h^{-\rho})$). When writing $\R_\sigma$, we mean that
that the set of real numbers $\R$ is the space where the variable
$\sigma$ varies.

We introduce the following operators in $\mathcal L
\big(L^2(\R_\sigma), L^2(\R_\sigma\times(0,h^{-\rho}))\big)$, and
in\break$\mathcal L
\big(L^2(\R_\sigma\times(0,h^{-\rho})),L^2(\R_\sigma)\big)$
respectively.

 The $L^2$-spaces that appear are equipped with the
standard $L^2$-norm. Recall the normalized eigenfunction $u_{0,h}$
introduced  in Remark~\ref{rem:es}. Define the {\it bounded}
operators $R_0^+$ and $R_0^-$  as follows:
\begin{align}
&(R_0^+v)(\sigma,\tau)=u_{0,h}(\tau)v (\sigma)\,,\label{eq:R0+}\\
&(R_0^-v)(\sigma)=\displaystyle\int_0^{h^{-\rho}}u_{0,h}(\tau)\,v(\sigma,\tau)\,d\tau\,.\label{eq:R0-}\\
\end{align}
Note that $R_0^+R_0^-$ is the orthogonal projection, in $L^2(\mathbb
R\times (0,h^{-\rho}))$, on the space $L^2(\mathbb R_\sigma) \otimes
\{ \mathbb R u_{0,h}\}$. The operator norms of $R_0^+$ and $R_0^-$
are both  equal to  $1$. The function $ R_0^+v\in
L^2(\R_\sigma\times(0,h^{-\rho}))$ will be extended by zero to a
function in $L^2(\R^2)$.

In the subsequent subsections, we will need two operators that
approximate $R_0^-$. Define the two operators $R_{0,h}^-$ and
$\widetilde R_{0,h}^-$ in $\mathcal
L\big(L^2((-h^{-\eta},h^{-\eta})\times
(0,h^{-\rho}));L^2((-h^{-\eta},h^{-\eta}))\big)$ as follows{:}
\begin{align}
&R_{0,h}^-v=\displaystyle\int_0^{h^{-\rho}}u_{0,h}(\tau)\,v(\sigma,\tau)\,\widetilde a(\sigma,\tau)\,d\tau\,,\label{eq:R0-h}\\
&\widetilde R_{0,h}^-v=\displaystyle\int_0^{h^{-\rho}}u_{0,h}(\tau)\,v(\sigma,\tau)\,\widetilde a(\sigma,\tau)^{-1}\,d\tau\,,\label{eq:R0-h'}
\end{align}
where  $\widetilde a$ is introduced in \eqref{eq:newa}.

 The
operators $R_{0,h}^-$ and $\widetilde R_{0,h}^-$ are approximations
of $R_0^-$ relative to the operator norm in $\mathcal
L\big(L^2((-h^{-\eta},h^{-\eta})\times
(0,h^{-\rho}));L^2((-h^{-\eta},h^{-\eta})\big)$ since, by
\eqref{atilde}, \eqref{atilde-1} and \eqref{estabc},
\begin{equation}\label{eq:R0h-R0}
\|R_{0,h}^--R_0^-\|\leq Ch^{\frac12-\rho}\quad{\rm and}\quad \|\widetilde R_{0,h}^--R_0^-\|\leq Ch^{\frac12-\rho}\,.
\end{equation}

\subsection{Approximate eigenfunctions for the harmonic
oscillator}

\subsubsection{Construction of the trial states and preliminary
estimates} ~\\
We use the eigenfunctions of the Laplacian $\mathcal L_h$ to
construct approximate eigenfunctions of the harmonic oscillator
$$H_{\rm harm}=-\frac{d^2}{d\sigma^2}+\frac{k_2}2\,\sigma^2\quad{\rm
in}~L^2(\R)\,.$$

Let $n\in\mathbb N$, $\eta\in(0,1/8)$,  and $v_{h,n}$ be a
normalized eigenfunction of  $\mathcal L_h$ associated with the
$n$-th eigenvalue $\mu_n(h)$. 

Recall  from Theorem \ref{thm:st} that $\mu_n(h)$  satisfies
$$\mu_n(h)\leq -h-\kappa_{\max}h^{3/2}+M_n h^{7/4}\,,$$
where $M_n>0$ is a constant. Define the functions $\psi_{h,n}$ and
$\widetilde\psi_{h,n}$  as follows:
\begin{align}&\psi_{h,n}(s,t)=
\chi\left(\frac{s}{h^{\frac18-\eta}}\right) \chi\left(\frac{t}{h^{\frac12-\rho}}\right)
v_{h,n}(s,t)\,,\label{eq:psih0}\\
&\widetilde\psi_{h,n}(\sigma,\tau)=h^{5/16} \psi_{h,n}(h^{1/8}\sigma,h^{1/2}\tau)\,.\label{eq:psih}
\end{align}
The function $\chi\in C_c^\infty(\R)$ satisfies $0\leq \chi\leq 1$
in $\R$, $\chi=1$ in $[-\frac12,\frac12]$ and the support of $\chi$
is in $[-1,1]$.

Notice that the function $\psi_{h,n}$ is defined in $\Omega$ by
means of the boundary coordinates in \eqref{BC}. %
\begin{pro}\label{lem:psih}
There holds:
\begin{enumerate}
\item $\displaystyle\left\|\mathcal
L_h\psi_{h,n}-\mu_n(h)\psi_{h,n}\right\|_2=  \mathcal
O(h^\infty)$\,;
\item $\displaystyle\left\|\psi_{h,n}-v_{h,n}\right\|_2 =\mathcal
O(h^\infty)$\,;
\item $\displaystyle\left\|\widetilde\psi_{h,n}-R_0^+R_0^-\widetilde\psi_{h,n}\right\|_2 =
o(1)$\,;
\item
$\displaystyle\|R_0^-\widetilde\psi_{h,n}\|_{L^2(\R)}=1+
o(1)$\,;\\
\item If $i\not=j$, then the functions $R_0^-\widetilde\psi_{h,i}$
and $R_0^-\widetilde\psi_{h,j}$ are almost orthogonal in the
following sense,
$$\langle
R_0^-\widetilde\psi_{h,i}\,,\,R_0^-\widetilde\psi_{h,j}\rangle_{L^2(\R)}=o(1)\,.$$
\end{enumerate}
\end{pro}
\begin{proof}
The assertions in (1) and (2) follow from Theorems~\ref{thm:dec} and
\ref{thm:dec-h}.

{\it Proof of (3):}\\
Thanks to the assertions in (1) and (2), we may write,
$$q_h(\psi_{h,n})\leq \mu_n(h)+\mathcal O(h^\infty)\,.$$
We express $q_h(\psi_{h,n})$ in boundary coordinates to write  (see
\eqref{eq:bc;qf})
\begin{multline}\label{eq:cor1}
\int_{-|\partial\Omega|/2}^{|\partial\Omega|/2}\int_0^{h^{\frac12-\rho}}
|h\partial_t \psi_{h,n}|^2\big(1-t\,\kappa(s)\big)\,dt
ds-h^{3/2}\int_{-|\partial\Omega|/2}^{|\partial\Omega|/2}\big| \psi_{h,n}(s,t=0)\big|^2d\sigma\\
\leq \mu_{n}(h) +\mathcal O(h^{\infty})\,.
\end{multline}
Using  that $\kappa(\cdot)$ is bounded, $\rho\in(0,1/2)$ and
Theorem~\ref{thm:dec}, we get
$$\left|\int_{-|\partial\Omega|/2}^{|\partial\Omega|/2}\int_0^{h^{\frac12-\rho}}|h \partial_t
\psi_{h,n}|^2\,  t\, \kappa(s)\,dt ds\right|\leq Chh^{\frac12-\rho}\,,$$ and, as a consequence of \eqref{eq:cor1},
$$
\int_{-|\partial\Omega|/2}^{|\partial\Omega|/2}\int_0^{h^{\frac12-\rho}}|h\partial_t
\psi_{h,n}|^2\,dt
ds-h^{3/2}\int_{-|\partial\Omega|/2}^{|\partial\Omega|/2}\big|\psi_{h,n}(s,t=0)\big|^2d\sigma
\leq \mu_{n}(h) +\mathcal O(h^{\frac32-\rho})\,.
$$
We apply the change of variables $(s,t)=(h^{1/8}\sigma,h^{1/2}\tau)$
and note that the support of the function $\widetilde\psi_{h,n}$ is
in $\{(\sigma,\tau)\in (-h^{-\eta},h^{\eta})\times [0,h^{-\rho})\}$.
In that way,  we obtain
\begin{equation}\label{eq:cor2}
h\left(\int_{-h^{-\eta}}^{h^{-\eta}}\int_0^{h^{-\rho}}|\partial_\tau
\widetilde\psi_{h,n}|^2\,d\tau
d\sigma-\int_{-h^{-\eta}}^{h^{-\eta}}\big|\widetilde\psi_{h,n}(\sigma,\tau=0)\big|^2d\sigma\right)
\leq \mu_{n}(h) +\mathcal O(h^{\frac32-\rho})\,.\end{equation}
 We have the following simple decomposition of $\widetilde\psi_{h,n}$,
$$\widetilde\psi_{h,n}(\sigma,\tau)=R_0^+R_0^-\widetilde\psi_{h,n}(\sigma,\tau)+\Big(\widetilde\psi_{h,n}-R_0^+R_0^-\widetilde\psi_{h,n}\Big)(\sigma,\tau)\,,
$$
with the property for every fixed $\sigma$,  the functions
$R_0^+R_0^-\widetilde\psi_{h,n}(\sigma,\cdot)$ and
$\Big(\widetilde\psi_{h,n}-R_0^+R_0^-\widetilde\psi_{h,n}\Big)(\sigma,\cdot)$
are orthogonal in $L^2((0,h^{-\rho}))$. Actually, it follows simply
from the definition of $R_0^+$ and $R_0^-$ that,
$$\forall~\sigma\in\R\,,\quad\int_0^{h^{-\rho}}R_0^+R_0^-\widetilde\psi_{h,n}(\sigma,\tau)\times
\Big(\widetilde\psi_{h,n}-R_0^+R_0^-\widetilde\psi_{h,n}\Big)(\sigma,\tau)\,d\tau=0\,.$$
%
%
Furthermore,  for every fixed $\sigma$, these functions are in the
domain of the operator $\mathcal H_{0,h}$ in
\eqref{eq:H0}-\eqref{eq:DomH0},
$R_0^+R_0^-\widetilde\psi_{h,n}(\sigma,\cdot)$ is in the first
eigenspace of the operator $\mathcal H_{0,h}$ and
$\Big(\widetilde\psi_{h,n}-R_0^+R_0^-\widetilde\psi_{h,n}\Big)(\sigma,\cdot)$
is in its orthogonal complement.
%
In that way, we get in light of the min-max principle and
Lemma~\ref{lem:1DL},  that a.e. in $\sigma$,
\begin{align*}
&\int_0^{h^{-\rho}}|\partial_\tau
\widetilde\psi_{h,n}(\sigma,\tau)|^2\,d\tau-|\widetilde\psi_{h,n}(\sigma,\tau=0)\big|^2\\
&\quad\geq
\lambda_1(\mathcal H_{0,h})\int_0^{h^{-\rho}}|
R_0^+R_0^-\widetilde\psi_{h,n}|^2\,d\tau+\lambda_2(\mathcal H_{0,h})\int_0^{h^{-\rho}}|
\widetilde\psi_{h,n}-R_0^+R_0^-\widetilde\psi_{h,n}|^2\,d\tau\\
&\quad\geq \big(-1+o(1)\big)\int_0^{h^{-\rho}}|
R_0^+R_0^-\widetilde\psi_{h,n}|^2\,d \tau
\end{align*}
Integrating over $\sigma$ in $(-h^{-\eta}, h^{-\eta})$ and using
also   \eqref{eq:cor2} and the property   that $R_0^+R_0^-$ is an
orthogonal projector,  we get
\begin{align*}
\big(-h+ho(1)\big)h&\int_{-h^{-\eta}}^{h^{-\eta}}\int_0^{h^{-\rho}}|R_0^+R_0^-\widetilde\psi_{h,n}|^2\,d\tau
d\sigma\\
&\leq\big(\mu_n(h)+\mathcal
O(h^{\frac32-\rho})\big)\int_{-h^{-\eta}}^{h^{-\eta}}\int_0^{h^{-\rho}}|R_0^+R_0^-\widetilde\psi_{h,n}|^2\,d\tau
d\sigma\\
&\qquad+\big(\mu_n(h)
+\mathcal
O(h^{\frac32-\rho})\big)\int_{-h^{-\eta}}^{h^{-\eta}}\int_0^{h^{-\rho}}|\widetilde\psi_{h,n}-R_0^+R_0^-\widetilde\psi_{h,n}|^2\,d\tau
d\sigma\,.
\end{align*}
Using that $ \mu_{n}(h)\leq-h+ho(1)$ and that $\rho\in(0,\frac12)$,
we get
$$\int_{-h^{-\eta}}^{h^{-\eta}}\int_0^{h^{-\rho}}|\widetilde\psi_{h,n}-R_0^+R_0^-\widetilde\psi_{h,n}|^2\,d\tau
d\sigma  =
o(1)\int_{-h^{-\eta}}^{h^{-\eta}}\int_0^{h^{-\rho}}|R_0^+R_0^-\widetilde\psi_{h,n}|^2\,
d\sigma\,.$$ Thus we get (3).

{\it Proof of (4):}\\
From the assertions (2) and (3)  and the normalization of $v_{h,n}$,
we get
$$\|R_0^+R_0^-\widetilde\psi_{h,n}\|_2=1+o(1)\,.$$
Using the definition of $R_0^+$, we get
$$\|R_0^+R_0^-\widetilde\psi_{h,n}\|_2=\|R_0^-\widetilde\psi_{h,n}\|_2\,,$$
and the assertion in (4) follows.

{\it Proof of (5):}\\
The functions $v_{h,i}$ and $v_{h,j}$ are orthogonal  if $i\neq j$
in
$L^2(\Omega)$.\\
By  assertions  (2) and (3), we get that
$R_0^+R_0^-\widetilde\psi_{h,i}$ and
$R_0^+R_0^-\widetilde\psi_{h,j}$ are {\it almost} orthogonal, i.e.
their inner product is $o(1)$. By definition of $R_0^+$, we have
$$\langle
R_0^+R_0^-\widetilde\psi_{h,i},R_0^+R_0^-\widetilde\psi_{h,j}\rangle_{L^2(\R^2)}=\langle
R_0^-\widetilde\psi_{h,i},R_0^-\widetilde\psi_{h,j}\rangle_{L^2(\R)}\,,$$
and  the assertion in (5) follows.
\end{proof}

It results from \eqref{eq:R0h-R0} the following simple corollary of
Proposition~\ref{lem:psih}.

\begin{cor}\label{corol:psih}
There holds,
\begin{enumerate}
\item
$\displaystyle\|\widetilde
R_{0,h}^-\widetilde\psi_{h,n}\|_{L^2(\R)}=1+
o(1)$\,;\\
\item If $i\not=j$, then the  functions $\widetilde R_{0,h}^-\widetilde\psi_{h,i}$
and $\widetilde R_{0,h}^-\widetilde\psi_{h,j}$ are almost
orthogonal in the following sense,
$$\langle
\widetilde R_{0,h}^-\widetilde\psi_{h,i}\,,\,\widetilde R_{0,h}^-\widetilde\psi_{h,j}\rangle_{L^2(\R)}=o(1)\,.$$
\end{enumerate}
\end{cor}

\subsubsection{Useful identities}

Recall the definition of the operators $P_0$, $P_2$, $P_3$ and $Q_h$
in \eqref{eq:dec-Lh}, together with the definition of the functions
$q_1$, $q_2$ and $q_3$ in \eqref{eq:dec-Lh2}-\eqref{eq:dec-Lh1}.

We introduce the operator
\begin{equation}\label{eq:newLh}
\widetilde{\mathcal L}_h=-h^{3/4}\tilde
a^{-1}\partial_\sigma(\tilde a^{-1}\partial_\sigma)-\tilde
a^{-1}\partial_\tau(\tilde a\partial_\tau)\,,
\end{equation}
with $\widetilde a$ as in \eqref{eq:newa}.
%

\begin{lem}\label{lem:id}
There holds:
\begin{enumerate}
\item $-\partial_\tau^2u_{0,h}=-u_{0,h}+\epsilon_1(h)u_{0,h}$ and $\partial_\tau u_{0,h}=-u_{0,h}+\epsilon_1(h)u_{0,h}+\epsilon_2(\tau;h)$
where
$$|\epsilon_1(h)|\leq C\exp(-2h^{-\rho})\,,\quad|\epsilon_2(\tau;h)|\leq
C\exp(-h^{-\rho})\,,$$ and $C>0$ is a constant.
\item For all $n\in\mathbb N$,
$$R_{0,h}^-\widetilde{\mathcal
L}_h\widetilde\psi_{h,n}=(-1-h^{1/2}\kappa_{\max})R_{0,h}^-\widetilde\psi_{h,n}+h^{3/4}H_{\rm harm}\widetilde R_{0,h}^-\widetilde\psi_{h,n}+h^{3/4}r(\widetilde\psi_{h,n})\,,
$$
where $k_2=-\kappa''(0)>0$,
$$H_{\rm harm}=-\frac{d^2}{d\sigma^2}+\frac{k_2}2\,\sigma^2\,,$$
and $r(\widetilde\psi_{h,n})$ is a function of $\sigma$
satisfying for some constant $C>0$ and all sufficiently small
$h$,
$$\|r(\widetilde\psi_{h,n})\|_{L^2(\R)}\leq Ch^{\frac18-3\eta}\,.$$
\end{enumerate}
\end{lem}
\begin{proof}
The assertion in (1) follows from \eqref{eq:wh}, the definition of
$u_{0,h}$ in \eqref{u0h} and a straightforward computation.

The functions $u_{0,h}$ and $\widetilde\psi_{h,n}$ vanish at
$\tau=h^{-\rho}$ and they satisfy the following Robin condition at
$\tau=0$, $$\partial_\tau u_{0,h}=-u_{0,h}\quad{\rm and}\quad
\partial_\tau\widetilde\psi_{h,n}=-\psi_{h,n}\,.$$
Thus, an integration by parts gives, for a.e. $\sigma$,
$$\int_0^{h^{-\rho}}u_{0,h}(\tau)\,\widetilde
a^{-1}\partial_\tau\big(\widetilde
a\partial_\tau\,\widetilde\psi_{h,n}\big)\,\widetilde a\,d\tau=
\int_0^{h^{-\rho}}\widetilde a^{-1}\partial_\tau\big(\widetilde
a\partial_\tau u_{0,h}\big)\,\widetilde\psi_{h,n}\,\widetilde
a\,d\tau\,.$$
The integral in the left hand side can be transformed to a simpler
one by using the (point-wise) expression of the operator
$a^{-1}\partial_\tau\big(\widetilde a\partial_\tau\big)$ in
\eqref{eq:dec-Lh2}. This allows us to write
\begin{align}\label{eq:dec-t'}
&-\int_0^{h^{-\rho}}u_{0,h}(\tau)\,\widetilde
a^{-1}\partial_\tau\big(\widetilde
a\partial_\tau\,\widetilde\psi_{h,n}\big)\,\widetilde a\,d\tau\nonumber\\
&\qquad=\int_0^{h^{-\rho}}\Big\{\big(P_0+h^{1/2}P_2+h^{3/4}\frac{\kappa''(0)}{2!}\sigma^2\partial_\tau
+h^{7/8}q_1(\sigma,\tau)\partial_\tau\big)u_{0,h}\Big\}\,\widetilde \psi_{h,n}\,\widetilde a\,d\tau\,.
\end{align}
The definition of the operators $P_0$, $P_2$  in
 \eqref{defPj} and the
assertion in (1) give us
\begin{equation}\label{eq:eh0}
\big(P_0+h^{1/2}P_2+h^{3/4}\frac{\kappa''(0)}{2!}\sigma^2\partial_\tau\big)u_{0,h}=\big(-1-\kappa_{\max}h^{1/2}+h^{3/4}\frac{k_2}{2}\sigma^2\big)u_{0,h}+e_h\,,
\end{equation}
where $k_2=-\kappa''(0)>0$ and
\begin{equation}\label{eq:eh}
e_h(\tau)=\Big[1+h^{1/2}\kappa_{\max}+h^{3/4}\frac{k_2\sigma^2}2\Big]\epsilon_1(h)u_{0,h}(\tau)+\Big[h^{1/2}\kappa_{\max}+h^{3/4}\frac{k_2\sigma^2}2\Big]\epsilon_2(\tau;h)\,.
\end{equation}
Since the variable $\tau$ lives in $(0,h^{-\rho})$, $0<\rho<\frac12$
and $\kappa(\cdot)$ is bounded, we get that $\widetilde a=
1+\mathcal O(h^{\frac12-\rho})$. Now, the properties of the
expressions $\epsilon_1$ and $\epsilon_2$ in assertion (1) yield
\begin{equation}\label{eq:eh'}
\left|\int_0^{h^{-\rho}}e_n(\tau)\widetilde\psi_{h,n}(\sigma,\tau)\,\widetilde a d\tau\right|\leq \mathcal O(h^\infty)
\left(\int_0^{h^{-\rho}}|\widetilde\psi_{h,n}(\sigma,\tau)|^2\,\widetilde a d\tau\right)^{1/2}\,.
\end{equation}
Using the normalization of $\widetilde\psi_{h,n}$ in the $L^2$-norm
relative to the measure $\widetilde a\,d\sigma\,d\tau$, we get
further,
\begin{equation}\label{eq:eh''}
\int_\R\left|\int_0^{h^{-\rho}}e_n(\tau)\widetilde\psi_{h,n}(\sigma,\tau)\,\widetilde a d\tau\right|^2d\sigma\leq \mathcal O(h^\infty)
\int_\R\int_0^{h^{-\rho}}|\widetilde\psi_{h,n}(\sigma,\tau)|^2\,\widetilde a d\tau=\mathcal O(h^\infty)\,.
\end{equation}
We infer from \eqref{eq:dec-t'} and \eqref{eq:eh0} the following
identity
\begin{equation}\label{eq:dec-t}
-\int_0^{h^{-\rho}}u_{0,h}(\tau)\,\widetilde
a^{-1}\partial_\tau\big(\widetilde
a\partial_\tau\,\widetilde\psi_{h,n}\big)\,\widetilde a\,d\tau
=\big(-1-h^{1/2}\kappa_{\max}+h^{3/4}\frac{k_2}{2}\sigma^2\big)R_{0,h}^-\widetilde\psi_{h,n}+r_1(\widetilde \psi_{h,n})\,,
\end{equation}
where $R_{0,h}^-\widetilde\psi_{h,n}$ is defined as in
\eqref{eq:R0-h}, and the remainder term is a function of the
variable $\sigma$ and defined as follows:
\begin{equation}\label{eq:r1}
r_1(\widetilde \psi_{h,n})=\int_0^{h^{-\rho}}e_h(\tau)\widetilde\psi_{h,n}(\sigma,\tau)\,\widetilde a\,d\tau+
\int_0^{h^{-\rho}}\left(h^{7/8}q_1(\sigma,\tau)\partial_\tau u_{0,h}(\tau)\right)\widetilde\psi_{h,n}(\sigma,\tau)\,\widetilde a\,d\tau\,.
\end{equation}
In the support of $\widetilde\psi_{h,n}$, $|\sigma|\leq h^{-\eta}$
and $0\leq \tau\leq h^{-\rho}$. The choice of $\rho$ and $\eta$ is
such that $\rho=\frac18+\eta$ and $0<\eta<\frac18$. Thanks to
\eqref{eq:qi}, $|q_1(\sigma,\tau)|\leq
C(|\sigma|^3+h^{1/8}\tau)=\mathcal O(h^{-3\eta})$ in the support of
$\widetilde\psi_{h,n}$. By the Cauchy-Schwarz inequality and the
assertion in (1), we may write
$$\left|\int_0^{h^{-\rho}}\left(h^{7/8}q_1(\sigma,\tau)\partial_\tau u_{0,h}(\tau)\right)\widetilde\psi_{h,n}(\sigma,\tau)\,\widetilde
a\,d\tau\right|\leq
Ch^{{\frac78-3\eta}}\left(\int_{0}^{h^{-\rho}}|\widetilde\psi_{h,n}(\sigma,\tau)|^2\,\widetilde
a^2\,d\tau\right)^{1/2}\,.$$ The normalization of $\widetilde\psi_{h,n}$ and the estimate $\widetilde a=1+\mathcal O(h^{\frac12-\rho})$  yield
$$
\int_\R\left|\int_0^{h^{-\rho}}\left(h^{7/8}q_1(\sigma,\tau)\partial_\tau
u_{0,h}(\tau)\right)\widetilde\psi_{h,n}(\sigma,\tau)\,\widetilde
a\,d\tau\right|^2\,d\sigma=\mathcal O(h^{\frac78-3\eta})\,.$$
Inserting this estimate and the one in \eqref{eq:eh''} into
\eqref{eq:r1}, we get that the $L^2$-norm in $\R_\sigma$ of the
remainder term is estimated as follows:
\begin{equation}\label{eq:r1'}
\|r_1(\widetilde \psi_{h,n})\|_{L^2(\R)}\leq
C\,h^{\frac78-3\eta}\,.\end{equation}
Next we look at the term
\begin{multline}\label{eq:sterm}
\int_0^{h^{-\rho}}u_{0,h}(\tau)\,\widetilde
a^{-1}\partial_\sigma\big(\widetilde
a^{-1}\partial_\sigma\widetilde\psi_{h,n}\big)\,\widetilde a\,d\tau
=\int_0^{h^{-\rho}}u_{0,h}(\tau)\,\widetilde
a^{-1}\partial_\sigma^2\widetilde\psi_{h,n}\,d\tau\\
+\int_0^{h^{-\rho}}u_{0,h}(\tau)\,\big(\widetilde
a^{-1}\partial_\sigma\widetilde
a^{-1}\big)\partial_\sigma\widetilde\psi_{h,n}\,d\tau\,.
\end{multline}
Here we will need the operator $\widetilde R^-_{0,h}$ in
\eqref{eq:R0-h'}. Since the function $u_{0,h}$ is independent of the
variable $\sigma$, 
$$\int_0^{h^{-\rho}}u_{0,h}(\tau)\,\partial_\sigma^2\big(\widetilde
a^{-1}\widetilde\psi_{h,n}\big)\,d\tau=\partial_\sigma^2\widetilde R^-_{0,h}\widetilde\psi_{h,n}\,.$$
Using the simple identity
$$\widetilde
a^{-1}\partial_\sigma^2\widetilde\psi_{h,n}=\partial_\sigma^2\big(\widetilde
a^{-1}\widetilde \psi_{h,n}\big)- 2\partial_\sigma\widetilde
a^{-1}\partial_\sigma\widetilde\psi_{h,n}-\big(\partial_\sigma^2\widetilde
a^{-1}\big)\widetilde\psi_{h,n}\,,$$  we infer from \eqref{eq:sterm},
\begin{equation}\label{eq:dec-s}
-h^{3/4}\int_0^{h^{-\rho}}u_{0,h}\,\widetilde
a^{-1}\partial_\sigma\big(\widetilde
a^{-1}\partial_\sigma\widetilde\psi_{h,n}\big)\,\widetilde a\,d\tau
=-h^{3/4}\partial_\sigma^2\widetilde
R_{0,h}^-\widetilde\psi_{h,n}+h^{3/4} r_2(\widetilde\psi_{h,n})\,,
\end{equation}
where the remainder term is a function of the variable $\sigma$ and
defined as follows:
\begin{equation}\label{eq:r2}
r_2(\widetilde\psi_{h,n})=\int_0^{h^{-\rho}}u_{0,h}\Big[2\big(\partial_\sigma\widetilde
a^{-1}\big)\times\big(\partial_\sigma\widetilde\psi_{h,n}\big)+\big(\partial_\sigma^2\widetilde
a^{-1})\widetilde\psi_{h,n}\Big]\, d\tau\,.\end{equation} The
definition of $\widetilde a$ and the boundedness of the curvature
yield that, in the support of $\psi_{h,n}$,
$$|\partial_\sigma\widetilde
a^{-1}|=\mathcal O\big(h^{\frac12-\rho}h^{\frac18}\big)\quad{\rm
and}\quad |\partial_\sigma^2\widetilde a^{-1}|=\mathcal
O\big(h^{\frac12-\rho}h^{\frac14}\big)\,.$$ Applying the
Cauchy-Schwarz ineqaulity, using the normalization of $u_{0,h}$ and
the expansion $\widetilde a=1+\mathcal O(h^{\frac12-\rho})$, we get
$$
\|r_2(\widetilde\psi_{h,n})\|_{L^2(\R)}\leq
Ch^{\frac12-\rho}h^{\frac18}\|\partial_\sigma\widetilde\psi_{h,n}\|_2+Ch^{\frac12-\rho}h^{1/4}\|\widetilde\psi_{h,n}\|_2\,.
$$
Performing the change of variables
$(s,t)=(h^{1/8}\sigma,h^{1/2}\tau)$, we get (see \eqref{eq:psih}),
$$\|\partial_\sigma\widetilde\psi_{h,n}\|_2=h^{\frac18}\|\partial_s\psi_{h,n}\|_2
\quad{\rm and}\quad \|\widetilde\psi_{h,n}\|_2=\|\psi_{h,n}\|_2=1+o(1)\,.
$$
Thus, the remainder term satisfies
$$\|r_2(\widetilde\psi_{h,n})\|_{L^2(\R)}\leq
Ch^{\frac12-\rho}h^{1/4}\|\partial_s\psi_{h,n}\|_2+Ch^{\frac12-\rho}h^{1/4}\,.$$
We estimate the norm of $\partial_s\psi_{h,n}$ as in
Remark~\ref{rem:partial-s}. In that way we get
\begin{equation}\label{eq:r2'}
\|r_2(\widetilde\psi_{h,n})\|_{L^2(\R)}\leq \widehat C
h^{\frac58-\rho}\,.\end{equation} Also, by \eqref{eq:R0h-R0} and the definition of $\rho=\frac18+\eta$, we
have
\begin{equation}\label{eq:r2''}\big\|\big(R_{0,h}^{-}-\widetilde
R_{0,h}^-\big)\widetilde\psi_{h,n}\big\|_2\leq
Ch^{\frac12-\rho}=Ch^{\frac38-\eta}\,.\end{equation} We set
$r(\widetilde\psi_{h,n})=h^{-3/4}r_1(\widetilde\psi_{h,n})+r_2(\widetilde\psi_{h,n})+\frac{k_2}2\sigma^2\big(R_{0,h}^{-}-\widetilde
R_{0,h}^-\big)\widetilde\psi_{h,n}$. Thanks to \eqref{eq:r1'},
\eqref{eq:r2'}, \eqref{eq:r2''}, and the condition $0<\eta<\frac18$,
we have
$$\|r(\widetilde\psi_{h,n})\|_{L^2(\R)}=\mathcal
O(h^{\frac18-3\eta})\,.$$ On the other hand, we insert
\eqref{eq:dec-s} into \eqref{eq:sterm} then  collect the obtained
identity and \eqref{eq:dec-t}. By virtue of the definition of
$\widetilde{\mathcal L}_h$ in \eqref{eq:newLh}, we get  the identity
in Assertion (2).
\end{proof}

\subsubsection{Validating  the approximation}

In the next lemma, we show that the functions $\widetilde
R_{0,h}^-\widetilde\psi_{h,n}$ yield approximate eigenfunctions for
the harmonic oscillator
$$H_{\rm harm}=-\frac{d^2}{d\sigma^2}+\frac{k_2}2\sigma^2\,.$$

\begin{lem}\label{lem:Lh=mu}
Let $n\in\mathbb N^*$,
$\Lambda_n=\big(\mu_n(h)+h+h^{3/2}\kappa_{\max}\big)h^{-7/4}$ and
$0<\eta<\frac1{24}$. There holds,
$$\big\|\big(H_{\rm
harm}-\Lambda_n\big)\widetilde R_{0,h}^-\widetilde\psi_{h,n}\big\|_2\leq C\left(h^{\frac18-3\eta}+h^{\frac38-\eta}|\Lambda_n|\right)\big\|\widetilde R_{0,h}^-\widetilde\psi_{h,n}\big\|_2\,.$$
\end{lem}
Thanks to Remark~\ref{rem:mu1} and Theorem~\ref{thm:st}, we have
$\Lambda_n=\mathcal O(1)$, as $h\to0_+$.

\begin{proof}[Proof of Lemma~\ref{lem:Lh=mu}]
We perform the change of variables
$(s,t)=(h^{1/8}\sigma,h^{1/2}\tau)$ and use \eqref{eq:psih} to write
(also we use \eqref{eq:scalingLh} to express the operator $\mathcal
L_h$ in the boundary coordinates),
$$
\widetilde{\mathcal L}_h\widetilde\psi_{h,n} (\sigma,\tau) =h^{5/16}h^{-1}\mathcal
L_h\psi_{h,n} (s,t)\quad{\rm and}\quad \widetilde\psi_{h,n}(\sigma,\tau)=h^{5/16}\psi_{h,n}(s,t) \,.
$$
By Lemma~\ref{lem:psih}, we have
$$
\|\widetilde{\mathcal
L}_h\widetilde\psi_{h,n}-h^{-1}\mu_n(h)\widetilde\psi_{h,n}\|_2 =
\mathcal O(h^\infty)\,.
$$
Since the norm of the operator  $R_{0,h}^-$ in \eqref{eq:R0-h} is
equal to $1$, we have
$$\big\|R_{0,h}^-\big(\widetilde{\mathcal
L}_h\widetilde\psi_{h,n}-h^{-1}\mu_n(h)\widetilde\psi_{h,n}\big)\big\|_2\leq
\|\widetilde{\mathcal
L}_h\widetilde\psi_{h,n}-h^{-1}\mu_n(h)\widetilde\psi_{h,n}\|_2=
\mathcal O(h^\infty)\,. $$ Now, Lemma~\ref{lem:id} and
\eqref{eq:R0h-R0} tell us that
\begin{align}\label{eq:dec-Lh''}
\left\|H_{\rm harm}\widetilde R_{0,h}^-\widetilde\psi_{h,n}-\Lambda_n \widetilde R_{0,h}^-\widetilde\psi_{h,n}\right\|_{L^2(\R)}
&\leq Ch^{\frac18-3\eta}+C|\Lambda_n|\,\|R_{0,h}^-\widetilde\psi_{h,n}-\widetilde R_{0,h}^-\widetilde\psi_{h,n}\|_{L^2(\R)}\\
&\leq Ch^{\frac18-3\eta}+C|\Lambda_n|h^{\frac12-\rho}=Ch^{\frac18-3\eta}+C|\Lambda_n|h^{\frac38-\eta}\,.
\end{align}
Corollary~\ref{corol:psih} tells us that the norm of $\widetilde
R_0^-\psi_{h,n}$ is
 $1+ o(1)$. In that way, \eqref{eq:dec-Lh''}  finishes the proof of
Lemma~\ref{lem:Lh=mu}.
\end{proof}

\begin{proof}[Lower bounds and three-term asymptotics] ~\\
Let $ n\in\mathbb N^*$ be a fixed natural number. Let $W_n$ be the
space spanned by the functions $\{R_0^-\widetilde\psi_{h,j}\}_{1\leq
j\leq n}$. These functions are almost orthonormal, see
Lemma~\ref{lem:psih}. Thanks to Lemma~\ref{lem:Lh=mu} and the
assumption $0<\eta<\frac18$, we see that
$$\forall\phi\in W_n \,,\quad \langle H_{\rm Harm}\phi,\phi\rangle_{L^2(\R)}\leq
\left\{h^{-7/4}\big(\mu_n(h)+h+h^{3/2}\kappa_{\rm max}\big)(1+o(1))+o(1)\right\}\|\phi\|_2\,,$$
and the dimension of $W_n$ is equal to $n$. \\
By the min-max principle, we deduce that the $n$-th eigenvalue of
the harmonic oscillator $H_{\rm harm}$ satisfies:
$$(2n-1)\sqrt{\frac{k_2}2}\leq h^{-7/4}\big(\mu_n(h)+h+h^{3/2}\kappa_{\rm
max}\big)(1+o(1))+o(1)\,.$$ This gives the following lower bound of $\mu_n(h)$,
$$\mu_n(h)\geq -h-h^{3/2}\kappa_{\rm
max}+(2n-1)\sqrt{\frac{k_2}2}\,h^{7/4}+h^{7/4}o(1)\,.$$ Now, thanks to the upper bound in Theorem~\ref{thm:st}, we get
\begin{equation}\label{eq:asym-mu}
\mu_n(h)= -h-h^{3/2}\kappa_{\rm
max}+(2n-1)\sqrt{\frac{k_2}2}\,h^{7/4}+h^{7/4}o(1)\,.
\end{equation}
\end{proof}

\begin{proof}[Proof of Theorem~\ref{thm:HK'}] ~\\
Let $n\in\mathbb N$ and $M\in\mathbb N\cup\{0\}$. Thanks to
Theorem~\ref{thm:ub}, we know that there exists an eigenvalue
$\widetilde\mu_n(h)$ of the operator $\mathcal L_h$ satisfying  as
$h\to0_+$,
$$\widetilde\mu_n(h)=-h-h^{3/2}\kappa_{\rm
max}+(2n-1)\sqrt{\frac{k_2}2}\,h^{7/4}+h^{7/4}\sum_{j=0}^M\zeta_{j,n}h^{j/8}+\mathcal
O\big(h^{\frac78+\frac{M+1}8}\big)\,.$$ Thanks to \eqref{eq:asym-mu}, we observe that there exists $h_0>0$ such that, for all $h\in(0,h_0)$,
$$\widetilde\mu_n(h)=\mu_n(h)\,.$$
Thus, $\mu_n(h)$ satisfies the asymptotic expansion in
Theorem~\ref{thm:HK'}.
\end{proof}

\section{WKB constructions}\label{sec:wkb}

In this section, we shortly address the WKB construction in the
spirit of Bonnaillie-H\'erau-Raymond \cite{BHR} (see previously in the context of the Born-Oppenheimer approximation \cite{Lef}).

We start by recalling the expression in \eqref{eq:newLh} of the
operator $L_h$ in the boundary coordinates $(s,t)$. We will express
this operator in the re-scaled coordinates $(s,\tau)=(s,h^{-1/2}t)$.
{ For this purpose it is convenient} to introduce the operator
\begin{equation}\label{eq:hLh}
\widehat L_h=-\partial_\tau^2-h\,\hat a^{-2}\partial_s^2-\hat a^{-1}(\partial_\tau\hat a)\partial_\tau-h\,\hat a^{-1}
(\partial_s\hat a^{-1})\partial_s\,,
\end{equation}
where
$$\hat
a(s,\tau)=1-h^{1/2}\tau\kappa(s)\,,\quad
\partial_\tau\hat a(s,\tau)=-h^{1/2}\kappa(s)\,,\quad \partial_s\hat a^{-1}(s,\tau)=h^{1/2}\hat a^{-2}(s,\tau)\kappa'(s)\,.$$
In the $(s,\tau)$ coordinates, the operator $\mathcal L_h$ is
expressed as follows,
$$\mathcal L_h=h\,\widehat L_h\,.$$
We will abuse notation and identify $\hat a^{-1}$, $\hat a^{-2}$ and
$\hat a^{-3}$ with the series expansions in $\tau$,
\begin{multline*}
\hat a^{-1}(s,\tau)=\sum_{j=0}^\infty
h^{j/2}\tau^j(\kappa(s))^j\,,\quad \hat
a^{-2}(s,\tau)=\sum_{j=0}^\infty c_j h^{j/2}\tau^j(\kappa(s))^j\,,\\
\hat a^{-3}(s,\tau)=\sum_{j=0}^\infty d_j
h^{j/2}\tau^j(\kappa(s))^j\,,
\end{multline*}
where, for all $j\in\{0,1,\cdots\}$,
$$c_0=d_0=1\,,\quad c_j=\frac{1}{j!}\prod_{k=0}^{j-1}(k-2)\,,\quad
d_j=\frac{1}{j!}\prod_{k=0}^{j-1}(k-3)\quad (j\geq 1)\,.$$
This leads to the following expansion of the operator $\widehat
L_h$,
\begin{align*}
\widehat L_h&=-\partial_\tau^2-h\partial_s^2\\
&~~+2h^{3/2}\tau\kappa(s)\partial_s^2+h^{1/2}\partial_\tau+h^{3/2}\tau\kappa'(s)\partial_s\\
&~~-\sum_{j=1}^\infty c_j h^{\frac{j+2}2}\tau^j(\kappa(s))^j\partial_s^2+\sum_{j=1}^\infty h^{\frac{j+1}2}\tau^j(\kappa(s))^{j+1}\partial_\tau
-\kappa'(s)\sum_{j=1}^\infty h^{\frac{j+3}2}d_j\tau^j(\kappa(s))^j\partial_s\,.
\end{align*}
We use the idea {\it \`a la} Born-Oppenheimer \cite{BOp,Lef, Mar}
and look for a trial state $ \Psi^{WKB} (s,\tau,h)$ having the WKB
expansion\footnote{More correctly, this means that  $$
   \exp\left( \frac{\vartheta (s)}{h^\frac 14}\right)  \Psi^{WKB} (s,\tau,h) \sim \sum_{\ell=0}^\infty a_\ell (s,\tau) h^\frac{\ell}{4}\,.
 $$}
 $$
 \Psi^{WKB} (s,\tau,h) \sim  \exp\left( - \frac{\vartheta (s)}{h^\frac 14}\right) \left( \sum_{\ell=0}^\infty a_\ell (s,\tau) h^\frac{\ell}{4}\right)
 $$
 and a corresponding eigenvalue
 $$\mu^{WKB} \sim  \sum_{\ell=0}^\infty \mu_\ell h^\frac \ell 4\,,$$
 where $\vartheta(s)$ is a real-valued function to be determined. \\

 As
 we shall see, the analysis below will allow us to select
 $\mu_0,\mu_1,\cdots$ so that $\mu^{WKB}$ is an expansion of the
 ground state energy of the operator $\mathcal L_h$ (this in particularly means that we impose the Robin condition). The focus on the ground state energy is for the sake of simplicity. A slight modification of the method
 allows us to get  expansions of all the low-lying
 eigenvalues of the operator $\mathcal L_h$.

We will arrange the terms in the equation
\begin{equation}\label{eq:WKB}
\widehat L_h\Psi^{WKB}\sim\mu^{WKB}\, \Psi^{WKB}\,,\end{equation} in
the form of    power series in $h^{1/4}$ and select $a_\ell
(s,\tau)$ and $\mu_\ell$ by matching the terms with coefficient
 $h^{\ell/4}$.

We introduce the (formal) operator
$$\widehat L_h^\vartheta:=\exp \left( \frac{\vartheta (s)}{h^\frac 14}\right)  \,\widehat L_h\, \exp \left(- \frac{\vartheta (s)}{h^\frac 14}\right) \,,$$
and notice that, \eqref{eq:WKB} is (formally) equivalent to
\begin{equation}\label{eq:WKB'}
\widehat
L_h^\vartheta\left( \sum_\ell a_\ell (s,\tau) h^\frac{\ell}{4}\right)\sim\mu^{WKB}\left( \sum_\ell a_\ell (s,\tau) h^\frac{\ell}{4}\right)\,.\end{equation}
The next step is to (formally) expand the operator $\widehat
L_h^\vartheta$ as follows:
\begin{equation}\label{eq:wkbL}
 \widehat L_h^\vartheta = \sum_{\ell=0}^\infty Q_\ell^\vartheta \,h^\frac \ell 4\,.
\end{equation}
In order to get the expressions of the operators $Q_\ell^\vartheta$,
we do a straightforward (but long!) calculation and observe that
\begin{align*}
\widehat L_h^\vartheta =&-\partial_\tau^2+h^{\frac12}\big(\partial_\tau-\vartheta(s)^2\big)+h^{\frac34}\big(2\vartheta'(s)\partial_s+\vartheta''(s)\big)\\
&+h^{\frac44}\big(-\partial_s^2+c_3\tau^3\kappa(s)^3+\tau^3(\kappa(s))^5\big)\\
&+h^{\frac54}\big(2c_3\tau^3\kappa(s)^3\vartheta'(s)\partial_s-c_3\vartheta''(s)\kappa(s)^3-\tau\kappa'(s)\vartheta'(s)\big)\\
&+h^{\frac64}\big(-c_5\tau^5\kappa(s)^5\partial_s^2+c_4\tau^4\vartheta'(s)\kappa(s)^4+\tau^4\kappa(s)^5\big)\\
&+\sum_{j=3}^\infty h^{\frac{2j+1}4}\tau^{j+1}\kappa(s)^{j+1}\big(2c_{j+1}\vartheta'(s)\partial_s-c_{j+1}\vartheta''(s)+d_{j+2}\tau\vartheta'(s)\kappa'(s)\kappa(s)\big)\\
&+\sum_{j=4}^\infty h^{\frac{2j}4}\tau^{j+1}\kappa(s)^{j+1}\big(-c_{j+2}\tau\kappa(s)\partial_s^2-d_{j+3}\tau^2\kappa'(s)\kappa(s)^2\partial_s-c_{j+1}\vartheta'(s)\big)\\
&+\sum_{j=4}^\infty h^{\frac{2j}4}\tau^{j+1}\kappa(s)^{j+2}\partial_\tau\,.
\end{align*}
In that way, we find the following expressions of the operators in
\eqref{eq:wkbL},
\begin{align*}
 &Q_0^\vartheta= -\partial_\tau^2\,,\\
 &Q_1^\vartheta =0\,,\\
 &Q_2^\vartheta  =  \kappa(s) \partial_\tau  - \vartheta'(s)^2 \,,\\
 &Q_3^\vartheta = 2 \vartheta'(s)  \partial_s + \vartheta'' (s)\,,\\
 &Q_4^\vartheta = -\partial_s^2+c_3\tau^3\kappa(s)^3+\tau^3(\kappa(s))^4\partial_\tau\,,\\
 &Q_5^\vartheta=2c_3\tau^3\kappa(s)^3\vartheta'(s)\partial_s-c_3\vartheta''(s)\kappa(s)^3-\tau\kappa'(s)\vartheta'(s)\,,\\
 &Q_6^\vartheta=-c_5\tau^5\kappa(s)^5\partial_s^2+c_4\tau^4\vartheta'(s)\kappa(s)^4+\tau^4(\kappa(s))^5\partial_\tau\,,
\end{align*}
and for all $\ell>6$, we have when $\ell$ is even,
$$
Q_\ell^\vartheta=\tau^{\frac\ell2+1}\kappa(s)^{\frac\ell2+1}\big(-c_{\frac\ell2+2}\tau\kappa(s)\partial_s^2-d_{\frac\ell2+3}\tau^2\kappa'(s)\kappa(s)^2\partial_s
-c_{\frac\ell2+1}\vartheta'(s)+\kappa(s)\partial_\tau\big)\,,
$$
and when $\ell$ is odd,
$$
Q_\ell^\vartheta=\tau^{\frac{\ell-1}4+1}\kappa(s)^{\frac{\ell-1}4+1}\big(2c_{\frac{\ell-1}4+1}\vartheta'(s)\partial_s-c_{\frac{\ell-1}4+1}\vartheta''(s)
+d_{\frac{\ell-1}4+2}\tau\vartheta'(s)\kappa'(s)\kappa(s)\big)\,.$$
We return to the formal equation in \eqref{eq:wkbL}. The term with
coefficient $h^0$ yields the equation,
 $$
 (Q_0^\vartheta-\mu_0) a_0(s,\tau) =0\,.
 $$
The operator $Q_0^\vartheta$ is nothing else than the operator
$I\otimes \mathcal H_{00}$ on $L^2(\mathbb R_s\times \mathbb
R_{+,\tau})$. Hence we will choose $a_0$ and later the $a_j$ in the
domain of $Q_0^\vartheta$ in order to respect the Robin condition.

 This leads us naturally (considering the operator $\mathcal H_{0,0}$ introduced in \eqref{defH00})  to the choice
 $$
 \mu_0=-1\quad{\rm and}\quad a_0 (s,\tau) = \xi_0(s) u_0(\tau)\,.
 $$
 Since
$Q_1^\vartheta=0$, the term in \eqref{eq:wkbL} with coefficient
$h^{1/4}$ yields
 $$
 (Q_0^\vartheta -\mu_0) a_1(s,\tau)-\mu_1a_1(s,\tau) =0\,.
 $$
 This leads us to the natural choice $\mu_1=0$ and
 $$
 a_1(s,\tau)=\xi_1(s) u_0(\tau)\,.
 $$
We look at the term with coefficient $h^{1/2}$ to obtain
$$
 (Q_0^\vartheta-\mu_0) a_2 + (Q_2^\vartheta -\mu_2) a_0 =0\,.
 $$
 Remembering that $\mu_0=-1$ and $a_0(s,\tau)=\xi_0(s)u_0(\tau)$, we
 get
$$
 (Q_0^\vartheta+1) a_2 + u_0(\tau) (\kappa(s)- \vartheta'(s)^2-\mu_2) \xi_0(s) =0\,.
 $$
 Multiplying by $u_0$ and integrating over $\tau$  leads us to the eikonal equation
 \begin{equation}\label{eq:eik}
 - \kappa(s) - \vartheta'(s)^2 -\mu_2 =0\,.
 \end{equation}
 Consequently, we  take $\mu_2=-\kappa(0)$, impose the condition $\vartheta'(0)=0$  and determine $\vartheta$
 from \eqref{eq:eik}. In particular, we obtain by differentiating
 \eqref{eq:eik} two successive times,
 \begin{equation}\label{eq:v''}
\vartheta''(0)=\sqrt{-\frac12\kappa''(0)}\,.
 \end{equation}
We  take $a_2$ in the form
 $$
 a_2(s,\tau)=\xi_2(s) u_0(\tau)\,,
 $$
so that the term $(Q_0^\vartheta+1) a_2$ vanishes.

Now we look at the term with coefficient $h^{3/4}$ in
\eqref{eq:wkbL}. This yields
 $$
 (Q_0^\vartheta+1) a_3 + (Q_2^\vartheta-\mu_2) a_1 + (Q_3^\vartheta -\mu_3)a_0 =0\,.
 $$
Using \eqref{eq:eik}, we see that the term $(Q_2^\vartheta-\mu_2)
a_1$ vanishes.  Choosing
 $$
 a_3(s,\tau) = \xi_3(s) u_0(\tau)\,.
 $$
 forces the term $(Q_0^\vartheta+1) a_3$ to vanish as well. We end up with the equation
$$(Q_3^\vartheta -\mu_3)a_0 =0\,.$$
 Multiplying by $u_0$ and integrating over
$\tau$ gives us the first transport equation
 \begin{equation}\label{eq:tr}
  2 \vartheta'(s) \xi_0'(s) + (\vartheta'' (s) - \mu_3) \xi_0(s)=0\,.
 \end{equation}
 In order to respect the condition $\vartheta'(0)=0$, \eqref{eq:tr} leads us to choose
 $$\mu_3 = \vartheta''(0)=  \sqrt{ - \frac 12
 \kappa''(0)}\,.$$ We can determine $\xi_0(s)\not=0$  by solving
 \eqref{eq:tr} in a neighborhood of $s=0$ and find:
 $$
 \xi_0(s) = c_0 \exp\left( - \frac{\vartheta'' (s)-\vartheta'' (0)}{2 \vartheta'(s)}\right)\,,
 $$
 where $c_0\neq 0$ is a constant.

\subsubsection*{The iteration process} ~\\
In \eqref{eq:wkbL}, the term with coefficient $h^{\ell/4}$,
$\ell\geq 4$, leads us to the equation,
\begin{equation}\label{eq:gterm}
\sum_{k=0}^\ell\big(Q_k^\vartheta-\mu_k\big)a_{\ell-k}=0\,.
\end{equation}
We need to define $a_\ell$ and $\mu_\ell$ for all $\ell\geq 4$. We
will do this by iteration.

Let us start by examining the case $\ell=4$. Thanks to the
conditions in \eqref{eq:eik}, \eqref{eq:gterm} becomes
\begin{equation}\label{eq:l=4}
(Q_0^\vartheta+1)a_4+(Q_3^\vartheta-\mu_3)a_1+(Q_4^\vartheta-\mu_4)a_0=0\,.\end{equation}
Multiplying by $u_0(\tau)$ and integrating over $\tau$ leads us to
the equation,
\begin{equation}\label{eq:Pxi1}
2\vartheta'(s)\xi_1'(s)+(\vartheta''(s)-\mu_3)\xi_1(s)+F_1(s)-\mu_4 \xi_0(s)=0\,,\end{equation}
where
$$F_1(s)=\int_0^\infty u_0(\tau)\,Q_4^\vartheta
a_0(s,\tau)\,d\tau\,.$$
We select $\mu_4$ such that
\begin{equation}\label{eq:conF1}
F_1(0)-\mu_4\xi_0(0) =0\,, \end{equation} and impose the condition
$$\xi_1(0)=0\,.$$
Note that this choice is consistent with \eqref{eq:Pxi1}, thanks to
$\mu_3=\vartheta''(0)$ and $\vartheta'(0)=0$.

The assumption on the  curvature is that it attains its
maximal value at the unique point $s=0$. Notice that, by
\eqref{eq:eik}, for $s\in (-\frac{|\partial
\Omega|}{2},\frac{|\partial \Omega|}{2} )$, $\vartheta'$ does not
vanish except at $0$ and that
$$2\vartheta'(s)\sim\sqrt{2k_2}\,s\quad{\rm and}\quad
\vartheta''(s)\sim\sqrt{\frac {k_2}2}\,,$$ where
$k_2=-\kappa''(0)>0$.

Dividing \eqref{eq:Pxi1} by $2 \vartheta'$, we get a first order
ordinary equation with $C^\infty$ coefficients in $(-\frac{|\partial
\Omega|}{2},\frac{|\partial \Omega|}{2} )$  which admits a unique
explicit solution $\xi_1(s)$ (by the method of the variation of
constants), if we impose the condition $\xi_1(0)=0\,$.

 With
this condition the function
$v_4=(Q_3^\vartheta-\mu_3)a_1-(Q_4^\vartheta-\mu_4)a_0$ is, for any
$s$, orthogonal to the function $u_0$ in $L^2(\R_+)$.

We return to \eqref{eq:l=4} and define $a_4(s,\tau)$ by
$$a_4(s,\tau)=\xi_4(s)u_0(\tau)+(Q_0^\vartheta+1)^{-1}v_4(s,\tau)\,,$$
where  $\xi_4(s)$ will be selected later  in the definition of
$\mu_7$ and $a_7\,$.

 Clearly, $a_4(s,\tau)$ is of the form,
$$a_4(s,\tau)=\xi_4(s)u_0(\tau)+\sum_{k=1}^{n_4}\xi_{4,k}(s)u_{4,k}(\tau)\,.$$

Now we describe the iteration process. Let $\ell\geq 5$ and
$\xi_{k}(s)$, $k\in\{4,\cdots,\ell\}$, be smooth functions. Suppose
that, $(\xi_k(s))_{k=0}^{\ell-3}$ and for all
$M\in\{0,\cdots,\ell-1\}$, $\mu_M\in\R$ and
$$a_{M}=\xi_M(s)u_0(\tau)+\sum_{k=1}^{n_{M}}\xi_{M,k}(s)u_{M,k}(\tau)$$
satisfy
$$\sum_{k=0}^{M}\big(Q_k^\vartheta-\mu_k\big)a_{M-k}=0\,.$$
We want to select $\mu_\ell\in\R$ and
$$a_{\ell}=\xi_\ell(s)u_0(\tau)+\sum_{k=1}^{n_{\ell}}\xi_{\ell,k}(s)u_{\ell,k}(\tau)$$
such that \eqref{eq:gterm} is satisfied. We expand \eqref{eq:gterm}
as follows:
$$
\big(Q_0^\vartheta+1\big)a_{\ell}+\big(Q_3^\vartheta-\mu_3\big)a_{\ell-3}+G_{\ell-3} +\big(Q_{\ell}^\vartheta-\mu_\ell\big)a_0=0\,,
$$
where
$$
G_{\ell-3}(s,\tau)=\sum_{k=4}^{\ell-1}\big(Q_k^\vartheta-\mu_k)a_{\ell-k}\,.
$$
Multiplying by $u_0(\tau)$ and integrating over $\tau$ yields
\begin{equation}\label{eq:xiell}
2\vartheta'(s)\xi_{\ell-3}'(s)+(\vartheta''(s)-\mu_3)\xi_{\ell-3}(s)+F_{\ell-3}(s)-\mu_\ell \xi_0(s)=0\,,
\end{equation}
where
$$F_{\ell-3}(s)=\int_0^\infty\Big\{\big(Q_3^\vartheta-\mu_3\big)\left(\sum_{k=1}^{n_{M}}\xi_{\ell-3,k}(s)u_{\ell-3,k}(\tau)\right)
+G_{\ell-3}(s,\tau)+Q_{\ell}^\vartheta
a_0(s,\tau)\Big\}u_0(\tau)\,d\tau\,.$$ We select $\mu_\ell$ such that
$$
F_{\ell-3}(0)-\mu_\ell \xi_0(0)=0\,,
$$
and impose the following condition
$$\xi_\ell (0)=0\,.$$
We divide both sides of \eqref{eq:xiell} by $2\vartheta'(s)$ and get
a first order differential equation in $\xi_{\ell-3}(s)$. Under the
condition $\xi_{\ell-3}(0)=0$, this equation has a unique solution
$\xi_{\ell-3}(s)$ defined in $(-\frac{|\partial
\Omega|}{2},\frac{|\partial \Omega|}{2} )$.  Then we define
$a_\ell(s,\tau)$ as follows:
$$a_\ell(s,\tau)=\xi_\ell(s)u_0(\tau)+(Q_0^\vartheta+1)^{-1}(v_{\ell-3}(\sigma,\tau))\,,$$
where
$$
v_{\ell-3}(\sigma,\tau)=- \big(Q_3^\vartheta-\mu_3\big)a_{\ell-3}-G_{\ell-3} -\big(Q_{\ell}^\vartheta-\mu_\ell\big)a_0\,.
$$

\end{document}